\documentclass{amsart}

\usepackage{amsthm,amsmath,amsthm,amssymb,mathrsfs,graphicx,mathtools,relsize,hyperref,cleveref,enumerate,nameref}
\usepackage[utf8]{inputenc}
\usepackage[all]{xy}

\theoremstyle{theorem}
 \newtheorem{thm}{Theorem}[section]
 \newtheorem{prop}[thm]{Proposition}
 \newtheorem{lem}[thm]{Lemma}
 \newtheorem{cor}[thm]{Corollary}
 
 \newtheorem{prob}[thm]{Problem}
\theoremstyle{definition}
 \newtheorem{exm}[thm]{Example}
 \newtheorem{dfn}[thm]{Definition}
\theoremstyle{remark}
 \newtheorem{rem}[thm]{Remark}
 \numberwithin{equation}{section}
 
\renewcommand{\le}{\leqslant}
\renewcommand{\ge}{\geqslant}
\renewcommand{\setminus}{\smallsetminus}
\def\Aut{\text{\rm Aut}}
\def\End{\text{\rm End}}

\def\Prim{\text{\rm Prim}}
\def\id{\text{\rm id}}
\def\Ad{\text{\rm Ad}}

\def\Hom{\text{\rm Hom}}

\def\hull{\text{\rm hull}}
\def\supp{\text{\rm supp}}
\def\ker{\text{\rm ker}}
\def\Lie{\text{\rm L}}

\def\tor{\text{\sf tor}}
\def\acts{\curvearrowright}
\def\H{\mathcal{H}}

\def\K{\mathcal{K}}
\def\B{\mathcal{B}}
\def\U{\mathcal{U}}

\def\BH{\mathcal{B}(\mathcal{H})}
\def\UH{\mathcal{U}(\mathcal{H})}
\def\ind{\text{\rm ind}}
\def\supp{\text{\rm supp}}
\def\Td{T_{d}}
\def\T{T}
\def\I{\text{\rm I}}

\def\acts{\curvearrowright}
\newcommand{\norm}[1]{\left\lVert#1\right\rVert}
\newcommand{\normmax}[1]{\left\lVert#1\right\rVert_{\text{max}}}

\title[On the unitary representation theory of contraction groups]{On the unitary representation theory of locally compact contraction groups}
\date{\today}

\subjclass[2020]{20C25, 22D10, 22D12, 22D25, 20G05, 43A65.}
\keywords{Unitary representation, type I group, CCR group, scale group, contraction group, unipotent linear algebraic group, amenable group, groups acting on trees}
\thanks{The author acknowledges support from an Australian Government Research Training Program (RTP) scholarship and from the FWO and F.R.S.-FNRS under the Excellence of Science (EOS) program (project ID 40007542).}

\author[Max Carter]{Max Carter} 
\address{Max Carter \\
Institut de recherche en mathématique et physique \\
Chemin du Cyclotron 2 \\
boîte L7.01.02 \\
Université catholique de Louvain \\ 
1348 Louvain-la-Neuve \\
Belgique.}
\email{max.carter@uclouvain.be}

\begin{document}

\begin{abstract}
The unitary representation theory of locally compact contraction groups and their semi-direct products with $\mathbb{Z}$ is studied. We put forward the problem of completely characterising such groups which are type I or CCR and this article provides a stepping stone towards a solution to this problem. In particular, we determine new examples of type $\I$ and non-type-$\I$ groups in this class, and we completely classify the irreducible unitary representations of the torsion-free groups, which are shown to be type I. When these groups are totally disconnected, they admit a faithful action by automorphisms on an infinite locally-finite regular tree; this work thus provides new examples of automorphism groups of regular trees with interesting representation theory, adding to recent work on this topic. 
\end{abstract}
\maketitle


\section{Introduction}
The unitary representation theory of locally compact groups enjoys an extensive history with strong connections to many areas of mathematics including physics and number theory (\textit{c.f. }\cite{HC68,Lan89,Var08}). The subject dates back to the early 1900's, with the pioneering work of Peter-Weyl on the representation theory of compact groups \cite{PW27}, and the work of Pontryagin on the representation theory of locally compact abelian groups \cite{Pon34}. Over the intervening years, a substantial theory of unitary representations for general locally compact groups has been developed, amongst which the work of Mackey is most notable; see \cite{Mac76} for an overview of his work. The strong interplay between unitary representations and operator algebras is critical to the theory and a fruitful area of study in its own right \cite{Nai56,FD88a,FD88b,Dix77,BH20}.

In the unitary representation theory of locally compact groups, an important question is, when given a locally compact group $G$, can we determine all the irreducible unitary representations of $G$ up to unitary equivalence. The collection of all equivalences classes of such representations forms a set, called the \textit{unitary dual} of $G$, and is denoted by $\widehat{G}$. The set $\widehat{G}$ is equipped with a topology, called the \textit{Fell topology}, and a Borel structure, called the \textit{Mackey-Borel structure}. It is a consequence of an important result of Glimm \cite{Gli61}, referred to as \textit{Glimm's theorem}, that essentially we can classify $\widehat{G}$ if and only if one of the following equivalent conditions is satisfied: $G$ is \textit{type I}; the group C$^*$-algebra $C^*(G)$ is GCR; the Fell topology on $\widehat{G}$ is $T_0$; the Mackey-Borel structure on $\widehat{G}$ is countably separated \cite[Theorem 7.6]{Fol95}.

A locally compact group $G$ is called \textit{type I} if every unitary representation of $G$ generates a type I von Neumann algebra. There is a strong dichotomy between type I groups and those which are not type I. The unitary representation theory of type I locally compact groups is quite tame: every unitary representation of a type I locally compact group decomposes uniquely into a direct integral of irreducible unitary representations \cite[$\S$8.6.6]{Dix77}, so the classification of unitary representations of a type I locally compact group reduces to classifying the irreducible ones. Moreover, as already alluded to, type I groups are, in some sense, those groups whose irreducible unitary representations can be classified in a `measurable' way. For locally compact groups which are not type I, however, pathological things occur: the prior mentioned direct integral decomposition is non-unique in a strong sense (\textit{c.f.}\! \cite[Theorem 7.41]{Fol95} and \cite[$\S$8.7]{Dix77}), and consequently, the study of general unitary representations can no longer be reduced to studying irreducible ones. Also, the complexity of classifying all irreducible unitary representations of a group which is not type I is extremely high, and in most cases leaves the problem of classifying all irreducible unitary representations of the group completely out of reach (see \cite{Tho15,TT19} for example, also refer to Glimm's theorem again).

It is thus a fundamental problem in unitary representation theory to determine which groups are type I. A substantial amount of work has been done on this problem over many decades and a comprehensive but non-exhaustive list of type I locally compact groups can be found in \cite[Thm 6.E.19, Thm 6.E.20, Thm 7.D.1]{BH20}. 

In recent years, the problem of determining which groups of automorphisms of regular trees are type I has been an important problem in the representation theory of totally disconnected locally compact (tdlc) groups \cite{Neb99,Ama03,Cio18,HR19,CKM21,Sem22,Sem23}. Automorphism groups of regular trees are often described by the experts as a `microcosm' for the theory of (simple) tdlc groups \cite{Cap18} which partially elucidates the importance of understanding the representation theory of such groups. For groups of automorphisms of regular trees that are unimodular and act transitively on the vertices of the tree, it was conjectured by Nebbia in \cite{Neb99} that the group C$^*$-algebra is CCR (see \cite[Page 207]{Fol95} for definition), which is stronger than the property of the group being type I, if and only if the group acts transitively on the boundary of the tree. Nebbia resolved the `only if' direction in the same paper, however, the `if' direction still remains open. Recent work by Semal in \cite{Sem22} suggests that a stronger version of the `if' direction may even hold. More recently than Nebbia's work, it was conjectured in \cite{HR19} that for non-amenable automorphism groups of regular trees that act `minimally' on the tree, the group is type I if and only if it acts transitively on the boundary of the tree. Parts of the conjecture were resolved in \cite{HR19} and improved upon in \cite{CKM21}. The conjecture, however, also still remains open.

The research described in the above paragraph mostly concerns automorphism groups of trees that are non-amenable, act (highly) transitively on the boundary of the tree, and in some cases satisfy Tits' independence property or one of its variants \cite{Tits70,BEW15}. At the time of writing this article, the (non-compact) amenable automorphism groups of trees have not received as much attention as their non-amenable counterparts, at least in terms of the study of their unitary representations. The only article that the author is aware of that studies unitary representations of non-compact amenable automorphism groups of regular trees is a short note by Nebbia \cite{Neb90}.

The present article arose out of a desire to understand, on one hand, more about the unitary representation theory of amenable automorphism groups of regular trees, and on the other hand, begin research on the unitary representation theory of the class of \textit{scale groups} \cite{Wil20} which have recently risen to importance in the theory of tdlc groups. A scale group is a closed subgroup of the automorphism group of a regular tree that fixes a boundary point and acts transitively on the vertices of the tree; such groups are necessarily amenable \cite[Chp I, Thm 8.1]{FTN91}. It is an important fact that every tdlc group with non-trivial scale function has a subquotient isomorphic to a scale group \cite{BW04}, so scale groups have connections with the broader theory of tdlc groups, and they also have many connections with the theory of self-replicating groups \cite{Hor15,Wil20}.

At the present time, completely characterising the type I scale groups and their irreducible unitary representations seems to be out of reach, so in this article we investigate the unitary representation theory of a particular subclass of scale groups which we call \textit{contractive scale groups}. Every scale group is a semi-direct product of the form $N \rtimes \langle t \rangle$, where $N$ is the normal subgroup of automorphisms that fix at least one vertex, and $t$ is an automorphism of the tree that translates towards the fixed boundary point \cite[Proposition 2.2]{Wil20}. A scale group $N \rtimes \langle t \rangle$ is \textit{contractive} if for all $x \in N$, $t^n x t^{-n} \rightarrow \id$ as $n \rightarrow \infty$. 

We abstract the above setup as follows so that the action on the tree is no longer present; we do this for simplification as our methods in this article do not require the tree-action and they rely soley on the already known structure theory of these groups. A \textit{locally compact contraction group} is a pair $(N, \alpha)$, where $N$ is a locally compact group, and $\alpha \in \Aut(N)$ is a contractive automorphism; see \cite{GW10,GW21,GW21b} for more details on the theory of contraction groups and their connections to the theory of (td)lc groups. Given any locally compact contraction group $(N,\alpha)$, we can look at the corresponding semi-direct product $N \rtimes_\alpha \mathbb{Z}$. If $N$ is totally disconnected, then it turns out that the group $N \rtimes_\alpha \mathbb{Z}$ admits a faithful action by automorphisms on a regular tree with respect to which it is a contractive scale group; this is a consequence of the \textit{tree representation theorem} \cite[Theorem 4.1]{BW04}. We thus get a correspondence between tdlc contraction groups $(N,\alpha)$ and contractive scale groups $N \rtimes_\alpha \mathbb{Z}$. 

The problem that this article works towards solving is the following.

\begin{prob}
Let $(N,\alpha)$ be a locally compact contraction group. Determine when $N$, or its semi-direct product $N \rtimes_\alpha \mathbb{Z}$, is type I or CCR.
\end{prob}

The main result of this article is the following theorem which characterises a large class of type I contractive scale groups and shows that they are never CCR groups. This is a new situation in the context of groups acting on trees, as for many of the known type I automorphism groups of regular trees, they are shown to be type I by showing that they satisfy the stronger CCR property. This is, however, expected, as these scale groups are, in some sense, totally disconnected analogues of the real $ax+b$ group and the representation theories of these groups are analogous in many ways. We also remark that the result \cite[pg.\ 240, Cor.\ 2]{DM76} implies that the left-regular representation of the groups in the following theorem are type I when $N$ is CCR, however, this does not imply that the groups themselves are type I.

\begin{thm}\label{thm:mainthm}
Let $(N,\alpha)$ be a locally compact contraction group and $G := N \rtimes_\alpha \mathbb{Z}$. The group $G$ is not CCR. Furthermore, if $N$ is assumed to be CCR, then the following hold:
\begin{enumerate}[(i)]
   \item $G$ is type I; \
   \item Let $X$ be a cross-section of the non-trivial orbits of the action $\mathbb{Z} \acts \widehat{N}$. Then $\widehat{G} = \{ \ind_N^{G} \pi : \pi \in X \} \cup \widehat{\mathbb{Z}}$.
\end{enumerate}
\end{thm}

It is shown in \cite{GW10} that every locally compact contraction group $(N,\alpha)$ is of the form
\begin{displaymath} N \cong N_0 \times N_{p_1} \times \cdots \times N_{p_n} \times \tor(N) \end{displaymath}
where $N_0$ is the connected component of the identity in $N$, $N_{p_i}$ is a nilpotent $p_i$-adic Lie group for each $i$ where $p_1, \dots, p_n$ are distinct primes, and $\tor(N)$ the subgroup of torsion elements of $N$. Each of the factors in this decomposition are invariant under the action of $\alpha$. Furthermore, it is known that $N_0$ is a connected simply-connected nilpotent real Lie group \cite{Sie86}, each of the $N_{p_i}$ are the $\mathbb{Q}_{p_i}$-rational points of a unipotent linear algebraic group over $\mathbb{Q}_{p_i}$ \cite{Wan84}, and $\tor(N)$ is totally disconnected \cite{GW10}. The Kirillov orbit method \cite{Kir59,Kir62,Kir04} provides a classification of the irreducible unitary representations of $N_0$ and shows that this group is CCR. Also, an extension of the Kirillov orbit method to unipotent linear algebraic groups over $\mathbb{Q}_p$ by Moore \cite{Mor65} gives a classification of the irreducible unitary representations of each of the $N_{p_i}$ and shows that they are CCR groups too. It thus follows that all torsion-free locally compact contraction groups are CCR and hence Theorem \ref{thm:mainthm} applies to these groups. In Theorem \ref{thm:kirclass} of this article, we provide an analogue of the Kirillov orbit method for scale groups of the form $N\rtimes_\alpha \mathbb{Z}$, where $(N,\alpha)$ is a torsion-free locally compact contraction group, and completely classify the irreducible unitary representations of these semi-direct products in terms of `coadjoint orbits'.

Much less is known about both the group-theoretic and the representation-theoretic properties of torsion locally compact contraction groups in comparison to the torsion-free case. It is shown in \cite{GW21} that locally pro-$p$ torsion locally compact contraction groups are nilpotent and their composition series, which are known to have finite length, are described in \cite{GW10} (we note that these result hold in the torsion-free case too). Also, in the torsion case we see non-type-I phenomena occurring: it is shown in the proof of \cite[Proposition 5.8]{CKM21} that if $F$ is a finite simple non-abelian group, then the torsion locally compact contraction group $(\bigoplus_{\mathbb{Z}_{<0}} F) \times (\prod_{\mathbb{Z}_{\ge 0}} F)$ equipped with the right-shift automorphism is a non-type-I group. This, unfortunately, is about the extent of our general knowledge on torsion locally compact contraction groups and their representation theory at the current time. Also, as far as the author is aware, prior to the writing of this article, there are no other known classes of non-type-I torsion locally compact contraction groups. 

The final section of this article initiates research on understanding the unitary representation theory of torsion locally compact contraction groups. We use unique and distinct methods to \cite{CKM21} to show that a countable subset of an uncountable set of torsion locally compact contraction groups is non-type-I, providing a second distinct class of examples of non-type-I torsion locally compact contraction groups (see Theorem \ref{thm:nontypeigps}). We pose the question as to whether all uncountably many of these groups are non-type-I too. Some difficulties prevented us from solving this problem using the techniques contained in this article, but a follow up article with a second author will treat this problem using more algebraic techniques. We finish the article with some brief discussion regarding the representation theory of unipotent linear algebraic groups over $\mathbb{F}_p(\!(t)\!)$. A natural question leading on from this project is to determine which unipotent linear algebraic groups over $\mathbb{F}_p(\!(t)\!)$ are type I or CCR. This question has been investigated multiple times in the past by prominent authors \cite{How77,How77a,EK12}, however, the question remains unresolved. We give an elementary proof that the $n$-dimensional Heisenberg groups over $\mathbb{F}_p(\!(t)\!)$ are CCR groups. This also shows, in particular, that there exist CCR locally compact contraction groups with non-abelian torsion subgroups. 

Our arguments in this article rely heavily on the \textit{Mackey little group method} (sometime referred to as the \textit{Mackey machine}) and its generalisations to C$^*$-dynamical systems and crossed-product C$^*$-algebras \cite{Mac58,Mac76,FD88b,FD88a,Ros94,Wil07,KT12,CELY17}. Giving a very simplified description, the Mackey little group method allows one to do the following: given a separable locally compact group $G$ and a closed normal type I subgroup $N$ of $G$, provided certain technical assumption hold, one can construct all the factor unitary representations (and in particular irreducible unitary representations) of $G$ from factor unitary representations of certain subgroups of $G/N$ referred to as the `Little groups'. This process of constructing representations of $G$ from representations of subgroups of $G/N$ preserves irreducibility and the type of the representations and this is critical to our arguments. The method also relies heavily on analysing the orbits of the action $G \acts \widehat{N}$ (or of the related C$^*$-dynamical system $G \acts C^*(N)$) and topological properties of the quotient space $\widehat{N}/G$. 

We give a detailed summary of the Mackey little group method along with other related representation and operator-algebra theoretic concepts in the preliminaries section of the article. We provide an extensive preliminaries section as we anticipate that there will be a broad range of mathematicians interested in this work, some coming from group theory, while others coming from representation theory and/or operator algebras.


\section{Preliminaries}

In this section we collect the notation and results that will be most critical to understanding the article. It is assumed that the reader has some familiarity with the standard results of functional analysis \cite{Rud91,Con07}, topological group theory \cite{HR79,Fol95}, C$^*$-algebras and their representations \cite{Dix77}, and von Neumann algebras \cite{Dix81}.


\subsection{Graphs and groups}
Let $V\Gamma$ be a set and $E\Gamma \subseteq \{ \{ u,v \} : u,v \in V\Gamma \}$. The pair $\Gamma = (V\Gamma, E\Gamma)$ is a \textit{graph}. The elements of $V\Gamma$ are the \textit{vertices} and elements of $E\Gamma$ the \textit{edges}. For $v \in V\Gamma$, the cardinality of the set $E(v) := \{ e \in E\Gamma : v \in e \}$ is the \textit{degree} of $v$, denoted $deg(v)$, and we call $\Gamma$ \textit{locally-finite} if $deg(v)$ is finite for all $v \in V\Gamma$. A graph is \textit{infinite} if $V\Gamma$ has infinite cardinality.

Let $I$ be a countable indexing set and $(v_i)_{i \in I}$ a sequence in $V\Gamma$. The sequence is called a \textit{path} if $v_i \ne v_{i+1}$ and $\{v_i, v_{i+1}\} \in E\Gamma$ for all $i \in I$. If $(v_i)_{i \in I}$ is a path and $I$ has finite cardinality, then we define its \textit{length} to be $\lvert I \rvert -1$, otherwise, we say the path is a \textit{ray} if $I = \mathbb{N}$ or \textit{bi-infinite} if $I = \mathbb{Z}$. 

If $I := \{0, 1, \dots, n\}$ and $(v_i)_{i \in I}$ is a path, the vertices $v_0$ and $v_n$ are the \textit{endpoints} of the path, and we say that the path is a \textit{cycle} if $v_0 = v_n$. For $v \in V\Gamma$, an edge of the form $\{v,v\}$ is called a \textit{loop}. The graph $\Gamma$ is \textit{connected} if for any two vertices $u,v\in V\Gamma$, there is a path in $\Gamma$ with endpoints $u$ and $v$. A connected graph with no loops or cycles is a \textit{tree}. The infinite tree in which every vertex has degree $d$, called the \textit{$d$-regular tree}, or a \textit{regular tree} for short, is denoted by $\Td$.

Fix an infinite tree $T = (VT, ET)$. Two rays in $T$ are said to be \textit{equivalent} if their intersection is also a ray. It can be checked that this is an equivalence relation on the set of all rays in $T$. The set of equivalence classes of rays under this equivalence relation is called the \textit{boundary} of $T$ and denoted by $\partial T$. The elements of $\partial T$ are the \textit{ends} (or \textit{boundary points}) of $T$. 


If $\Gamma$ and $\Lambda$ are graphs, a \textit{homomorphism} $\phi: \Gamma \rightarrow \Lambda$ is a function $\phi: V\Gamma \rightarrow V\Lambda$ such that for any $\{u,v\} \in E\Gamma$, $\{ \phi(u), \phi(v) \} \in E\Lambda$. An \textit{automorphism} of $\Gamma$ is a bijective graph homomorphism of $\Gamma$ to itself. The group of automorphisms of a graph $\Gamma$ is denoted by $\Aut(\Gamma)$. Throughout this paper, it will be assumed that $\Aut(\Gamma)$ has the permutation topology: the topology whose base neighbourhoods are the collection of sets of the form $\mathcal{U}(g_0,\mathcal{F}) := \{ g \in \Aut(\Gamma) : g(x) = g_0(x) \; \text{for all} \; x \in \mathcal{F} \}$, with $g_0$ ranging over all elements of the group $\Aut(\Gamma)$ and $\mathcal{F}$ ranging over all finite subsets of $V\Gamma$. It is assumed that any subgroup of $\text{Aut}(\Gamma)$ has the induced topology. If $\Gamma$ is locally-finite i.e.\ every vertex in $\Gamma$ has finite degree, then $\Aut(\Gamma)$ is a totally disconnected locally compact second countable group when endowed with the permutation topology.

Given a group $G$ acting on a set $X$, for any subset $Y \subseteq X$, define $G_{Y}:=\{g \in G \mid g(Y)=Y \}$, the \textit{stabiliser} of $Y$ under the action of $G$. If $Y = \{ y \}$ then we will write $G_y$ instead of $G_{\{y\}}$. Similarly, $\text{Fix}_{G}(Y) := \{ g \in G \mid g(y) = y \; \forall y \in Y \}$ and is called the \textit{fixator} of $Y$ in $G$. For $x \in X$, $G(x) := \{ g(x) \mid g \in G \}$ denotes the \textit{orbit} of $x$ under $G$. 


\subsection{Scale groups}
We begin by defining scale groups. The primary reference on scale groups is \cite{Wil20}.

\begin{dfn}[Scale group]
Let $d \in \mathbb{N}_{>2}$ and $\omega \in \partial \Td$. A subgroup $G \le \Aut(\Td)$ which is closed, acts transitively on the vertices of $\Td$ and fixes $\omega$ is called a \textit{scale group}.
\end{dfn}

Let $N$ be the subgroup of the scale group $G$ consisting of the elements of $G$ that fix at least one vertex of $\Td$.The subgroup $N$ is normal in $G$.  The orbits of the action $N \acts V\Td$ are called the \textit{horocycles} of the end $\omega$. The subgroup $N$ will from now on be referred to as the \textit{horocycle stabiliser} subgroup of $G$.

\begin{prop}\cite[Proposition 2.2]{Wil20}
Let $G \le \Aut(\Td)$ be a scale group fixing the end $\omega \in \partial\Td$ and $N$ the subgroup of $G$ that stabilises the horocycles of $\omega$. There exists a translation $t \in G$, translating towards $\omega$, such that $G = N \rtimes \langle t \rangle$. 
\end{prop}

\begin{dfn}
The scale group $G = N \rtimes \langle t \rangle$ is \textit{contractive} if $t^n n t^{-n} \rightarrow \id_N$ as $n \rightarrow \infty$ for all $n \in N$.
\end{dfn}


\subsection{Contraction groups}
The primary references on contraction groups are \cite{GW10,GW21,GW21b}.
\begin{dfn}[Contraction group]
Let $N$ be a topological group and $\alpha$ an (bicontinuous) automorphism of $N$. The pair $(N, \alpha)$ is called a \textit{contraction group} if for all $g \in N$, $\alpha^n(g) \rightarrow \id_N$ as $n \rightarrow \infty$. If $N$ satisfies a property P (e.g.\ P = locally compact, totally disconnected, torsion etc.) then $(N,\alpha)$ will be called a \textit{P contraction group}.
\end{dfn}

We now list some important examples of contraction groups that will be used in this article and set some corresponding notation.

\begin{exm}\label{exm:abcont}
\begin{enumerate}[(i)]
   \item The additive group $(\mathbb{Q}_p,+)$ with automorphism multiplication by $p$.
   \item The additive group $(\mathbb{F}_{p^n}(\!(t)\!),+)$ of formal Laurent series with coefficients in the field $\mathbb{F}_{p^n}$ is an abelian contraction group with respect to the automorphism $\alpha_t: \mathbb{F}_{p^n}(\!(t)\!) \rightarrow \mathbb{F}_{p^n}(\!(t)\!), x \mapsto tx$. \
   \item Generalising the previous example, let $F$ be a finite group. The group $F(\!(t)\!):=(\bigoplus_{\mathbb{Z}_{<0}} F) \times (\prod_{\mathbb{Z}_{\ge 0}} F)$ equipped with the automorphism $\alpha_{\text{rs}}: F(\!(t)\!) \rightarrow F(\!(t)\!), (x_i)_{i \in \mathbb{Z}} \mapsto (x_{i+1})_{i \in \mathbb{Z}}$  is a locally compact contraction group. A special case of this that will be used in the next section is when $F=C_{p^n}$, the cyclic group of order $p^n$ where $p$ is a prime. \
   \item Following the notation and results in \cite{GW21}, let $\mathbb{P}$ denote the set of all primes and $\mathbb{Q}_\infty := \mathbb{R}$. For $p \in \mathbb{P} \cup \{ \infty \}$, let $\Omega_p \subseteq \mathbb{Q}_p[X]$ be the set of all monic irreducible polynomials whose roots are all non-zero and have absolute value $<1$ in some algebraic closure $\overline{\mathbb{Q}_p}$ of $\mathbb{Q}_p$. For $f \in \Omega_p$ and $n \in \mathbb{N}$, define $E_{f^n} := \mathbb{Q}_p[X]/f^n\mathbb{Q}_p[X]$.  The group $E_{f^n}$ is a locally compact contraction group when equipped with the automorphism $\alpha_{f^n}: E_{f^n} \rightarrow E_{f^n}, g + f^n\mathbb{Q}_p[X] \mapsto Xg+f^n \mathbb{Q}_p[X]$. Note that $E_{f^n}$ is isomorphic to a direct sum of finitely many copies of $\mathbb{Q}_p$. \
   \item The group $U_n(\mathbb{Q}_p)$ of $n$-dimensional unipotent matrices over $\mathbb{Q}_p$ equipped with the automorphism which is conjugation by the dialgonal matrix $\text{diag}(p,p^2,\dots,p^n)$.
\end{enumerate}
\end{exm}

The following is an important structure theorem about locally compact contraction groups that will be used throughout the article. Note that if $(N,\alpha)$ is a contraction group and $M \subseteq N$, then we call $M$ \textit{$\alpha$-stable} if $\alpha(M) = M$. Similarly, $M$ is called \textit{fully invariant} (resp.\ \textit{topologically fully invariant}) if $f(M) \subseteq M$ for every set map $f:N \rightarrow N$ (resp.\ every continuous set map $f: N \rightarrow N$).

\begin{thm}\cite[Theorem A]{GW21}\label{thm:contgp}
For each locally compact contraction group $(N,\alpha)$, we have:
\begin{enumerate}[(i)]
   \item Let $N_0$ denote the connected component of the identity in $N$. Then $N$ has a unique closed normal subgroup $N_{td}$ such that $N = N_0 \times N_{td}$ internally as a topological group. The subgroup $N_{td}$ is totally disconnected, $\alpha$-stable, and topologically fully invariant. \
   \item The set $\tor(N)$ of torsion elements of $N$ is a closed subgroup of $N$ and totally disconnected. There is a unique $n \in \mathbb{N} \cup \{0\}$, unique prime numbers $p_1 < \cdots < p_n$ and unique $p$-adic Lie groups $N_p \ne \{\id\}$ for $p \in \{p_1, \dots, p_n\}$ which are closed normal subgroups of $N$ such that 
   \begin{displaymath} N = N_0 \times N_{p_1} \times \cdots \times N_{p_n} \times \tor(N) \end{displaymath}
   internally as a topological group. Each $N_p$ is topologically fully invariant in $N$ and hence $\alpha$-stable. \
   \item If $\tor(N)$ is locally pro-nilpotent, then, for each prime number $p$, the set $\tor_p(N)$ of $p$-torsion elements of $N$ is a fully invariant closed subgroup of $N$ which is locally pro-$p$. Moreover, $\tor_p(N) \ne \{ \id_N \}$ for only finitely many $p$, say for $p$ among the prime numbers $q_1 < q_2 < \cdots < q_m$, and
   \begin{displaymath} \tor(N) = \tor_{q_1}(N) \times \cdots \times \tor_{q_m}(N) \end{displaymath}
   internally as a topological group.
\end{enumerate}
\end{thm}

\begin{rem}
Recall, as mentioned in the introduction, the group $N_0$ is always a connected simply-connected nilpotent real Lie group and the $N_{p_i}$ are unipotent linear algebraic groups over $\mathbb{Q}_p$. See \cite{GW21} for more details.
\end{rem}

The following result classifies all abelian locally compact contraction groups up to isomorphism. Consult Example \ref{exm:abcont} for the definition of the notation.

\begin{thm}\cite[Theorem E]{GW21}\label{thm:abeliancont}
Let $A$ be a locally compact abelian group and suppose that $\alpha$ is an automorphism of $A$ such that $(A,\alpha)$ is a contraction group. Then $(A,\alpha)$ is isomorphic to
\begin{displaymath} \bigoplus_{p \in \mathbb{P} \cup \{\infty\}} \bigoplus_{f \in \Omega_p} \bigoplus_{n \in \mathbb{N}} (E_{f^n},\alpha_{f^n})^{\mu(p,f,n)} \oplus \bigoplus_{p \in \mathbb{P}} \bigoplus_{n \in \mathbb{N}} (C_{p^n}(\!(t)\!),\alpha_\text{rs})^{\nu(p,n)} \end{displaymath}
as a contraction group, for uniquely determined $\mu(p,f,n) \in \mathbb{N} \cup \{0\}$ which are non-zero for only finitely many $(p,f,n)$, and uniquely determined $\nu(p,n) \in \mathbb{N} \cup \{0\}$ which are non-zero for only finitely many $(p,n)$. Conversely, all groups of the above form are locally compact abelian contraction groups.
\end{thm}

It is shown in \cite{GW10} that for any locally compact contraction group $(N,\alpha)$, there exists a composition series 
\begin{displaymath} 1 = N_0 \triangleleft \dots \triangleleft N_m = N \end{displaymath}
of $\alpha$-stable closed subgroups of $N$. The quotients $N_{i+1}/N_{i}$ are all \textit{simple contraction groups} in the following sense.

\begin{dfn}
A locally compact contraction group $(N,\alpha)$ is called a \textit{simple contraction group} if it has no $\alpha$-stable closed normal subgroups.
\end{dfn}

Glöckner and Willis have classified all of the simple locally compact contraction groups. Again, the reader should consult Example \ref{exm:abcont} for the definition of the notation in the following theorem.

\begin{thm}\cite{GW10,GW21}\label{thm:simpcont}
Let $(N,\alpha)$ be a simple locally compact contraction group. The following hold:
\begin{enumerate}[(i)]
   \item N is either connected or totally disconnected. If it is connected, then it is torsion-free, and if it is totally disconnected, then it is either torsion-free or torsion; \
   \item If $N$ is connected, then it is abelian, and $(N,\alpha)$ is isomorphic to $(E_{f^n},\alpha_{f^n})$ for some $f \in \Omega_{\infty}$ and $n \in \mathbb{N}$;
   \item If $N$ is totally-disconnected and torsion-free, then it is abelian, and $(N,\alpha)$ is isomorphic to $(E_{f^n},\alpha_{f^n})$ for some $f \in \Omega_{p}$, $n \in \mathbb{N}$ and $p \in \mathbb{P}$;
   \item If $N$ is totally-disconnected and torsion then $(N,\alpha)$ is isomorphic to $((\bigoplus_{\mathbb{Z}_{<0}} F) \times (\prod_{\mathbb{Z}_{\ge 0}} F),\alpha_{\text{rs}})$ for some finite simple group $F$.
\end{enumerate}
\end{thm}

The link between contraction groups, scale groups and groups acting on trees is the following theorem. Its proof is given by the tree representation theorem in \cite{BW04} and extensively uses the scale theory of totally disconnected locally compact groups. 

\begin{prop}\cite[Theorem 4.1]{BW04}
Let $(N,\alpha)$ be a totally disconnected locally compact contraction group. The group $N \rtimes_\alpha \mathbb{Z}$ admits a faithful action by automorphisms on a regular tree $\Td$ such that it is a scale group.
\end{prop}

It is clear by definition that any scale group of the form $N \rtimes_\alpha \mathbb{Z}$, where $(N,\alpha)$ is a totally disconnected locally compact contraction group, is a contractive scale group.

\subsection{Multipliers and representations}
In this subsection we follow the notation and terminology used in \cite{BK73,Kle74}. The reader can consult these reference for more details on the following.

Let $G$ be a locally compact second countable group. A \textit{multiplier} $\omega$ on $G$ is a measurable function $\omega: G \times G \rightarrow \mathbb{T}$ that satisfies the following properties:
\begin{enumerate}
\item[(M1)] $\omega(x,y)\omega(xy,z) = \omega(x,yz)\omega(y,z)$, for all $x,y,z \in G$;\
\item[(M2)] $\omega(x,e) = \omega(e,x) = 1$, for all $x \in G$. \
\end{enumerate}

Two multipliers $\omega_1$ and $\omega_2$ of $G$ are \textit{similar} if there exists a measurable function $\xi: G \rightarrow \mathbb{T}$ such that $\omega_1(x,y) = \xi(x)\xi(y)\xi(xy)^{-1} \omega_2(x,y)$ for all $x,y \in G$. The multiplier $\omega$ is \textit{normalized} if $\omega(x,x^{-1}) =1 $ for all $x \in G$. Every multiplier is similar to a normalized multiplier. For a multiplier $\omega$ on $G$, define $\omega^{(2)}(x,y) := \omega(x,y)\omega(y,x)^{-1}$ and  $S_\omega := \{ x \in G : \omega^{(2)}(x,y) = 1 \; \forall y \in G \}$. Note that $S_\omega$ is a closed subgroup of $G$. The multiplier $\omega$ is called \textit{totally skew} if $S_\omega$ is trivial.


Fix a multiplier $\omega$ of $G$. Given a Hilbert space $\H$, $\UH$ will denote the multiplicative group of unitary operators on $\H$ equipped with the strong operator topology. An \textit{$\omega$-unitary representation} of $G$, or \textit{$\omega$-representation} for short, is a pair $(\pi, \H_\pi)$, where $\H_\pi$ is a Hilbert space of arbitrary dimension and $\pi: G \rightarrow \U(\H_\pi)$ a function that satisfies:
\begin{enumerate}
\item[(MR1)] $\pi(x)\pi(y) = \omega(x,y)\pi(xy)$, for all $x,y \in G$; \
\item[(MR2)] $\pi: G \rightarrow \U(\H_\pi)$ is measurable.
\end{enumerate}

If $\omega$ is the trivial multiplier, then it turns out that $\pi$ is continuous \cite[Theorem 1]{Kle74}, and in this case $(\pi,\H_\pi)$ is a \textit{unitary representation} of $G$. The $\omega$-representation $\pi$ is called \textit{irreducible} if there does not exists any proper non-trivial closed invariant subspaces $\K \subseteq \H_\pi$. Given two $\omega$-representations $(\pi,\H_\pi)$ and $(\sigma, \H_\sigma)$, a bounded operator $T: \B(\H_\pi) \rightarrow \B(\H_\sigma)$ \textit{intertwines} $\pi$ and $\sigma$ if $T\pi(g) = \sigma(g) T$ for all $g \in G$. The vector space of all operators that intertwine $\pi$ and $\sigma$ will be denoted by $\Hom(\pi,\sigma)$. If $\pi = \sigma$, then $\Hom(\pi,\sigma)$ is a von Neumann algebra. The $\omega$-representations $(\pi,\H_\pi)$ and $(\sigma, \H_\sigma)$ are \textit{unitary equivalent} if there exists a unitary operator $U \in \Hom(\pi,\sigma)$, in which case, we write $\pi \simeq \sigma$. The set of all irreducible $\omega$-representations of the group $G$ up to unitary equivalence is denoted by $\widehat{G}^\omega$ (or by $\widehat{G}$ if $\omega$ is trivial). If P is a property of von Neumann algebras (e.g. type I, type II, type III, factor etc.) then $(\pi, \H_\pi)$ is called a \textit{P $\omega$-representation} if the von Neumann algebra generated by $\pi(G)$ is a P von Neumann algebra. If all the $\omega$-representations of $G$ are type I, then we call $\omega$ a \textit{type I multiplier}. 



\subsection{Group C$^*$-algebras and the Fell topology}

We now discuss group C$^*$-algebras and the Fell topology. It is assumed throughout this subsection that $G$ is a locally compact group equipped with a fixed Haar measure $\mu$. 

Given a unitary representation $(\pi, \H_\pi)$ of $G$, we can extend $\pi$ to a $*$-representation of $L^1(G)$ as follows: for $f \in L^1(G)$ and $\xi \in \H_\pi$, we define $\pi(f)\xi := \int_{G} f(g)\pi(g)\xi \; d\mu(g)$, where we interpret the integral as a Bochner integral \cite[Appendix 3]{Fol95}. This sets up a correspondence between unitary representations of $G$ and non-degenerate $*$-representations of $L^1(G)$ (here, a $*$-representation $\pi$ of $L^1(G)$ is non-degenerate if the only vector $\xi \in \H_\pi$ satisfying $\pi(L^1(G))\xi = \{0\}$ is the trivial vector).

\begin{thm}\cite[$\S$13.3]{Dix77}
Let $G$ be a locally compact group. There is a one-to-one correspondence between the unitary representations of $G$ and non-degenerate $*$-representations of $L^1(G)$. This correspondence preserves irreducibility and unitary equivalence.
\end{thm}

Define a norm $\normmax{\, \cdot \,}$ on $L^1(G)$, which we call the \textit{maximal C$^*$-norm}, by $\normmax{f} := \sup\{ \norm{\pi(f)}_{\text{op}} : \pi \in \widehat{G} \}$, for $f \in L^1(G)$. It can be checked that this is indeed a norm on $L^1(G)$ and the completion of $L^1(G)$ with respect to this norm is a C$^*$-algebra. The completion is the \textit{group C$^*$-algebra} of $G$ and is denoted by $C^*(G)$.

Since $L^1(G)$ is dense in $C^*(G)$ and every non-degenerate $*$-representation of $L^1(G)$ is continuous with respect to the maximal C$^*$-norm, every non-degenerate $*$-representation of $L^1(G)$ extends uniquely to a non-degenerate $*$-representation of $C^*(G)$, and by the prior discussion, the unitary representations of $G$ are thus in bijective correspondence with the non-degenerate $*$-representations of $C^*(G)$. 


%
%
Two $*$-representations $(\pi, \H_\pi)$ and $(\sigma, \H_\sigma)$ of a C$^*$-algebra $A$ are called \textit{unitary equivalent} if there exists a unitary operator $U: \H_\pi \rightarrow \H_\sigma$ such that $U\pi(x) = \sigma(x) U$ for all $x \in A$, identical to the definition for groups. The collection of equivalence classes of irreducible $*$-representations of a C$^*$-algebra $A$ up to unitary equivalence forms a set, denoted $\widehat{A}$, and is called the \textit{unitary dual} of $A$.

\begin{dfn}
Let $A$ be a C$^*$-algebra and $I$ an ideal of $A$. The ideal $I$ is called \textit{primitive} if there exists an irreducible $*$-representation $(\pi, \H_\pi)$ of $A$ such that $I = \text{ker}(\pi)$. The collection of all primitive ideals of $A$ is called the \textit{primitive ideal space} and is denoted by $\Prim(A)$.
\end{dfn}

We denote by $\kappa$ the canonical map $\kappa : \widehat{A} \rightarrow \Prim(A), \pi \mapsto \ker(\pi)$. The set $\Prim(A)$ also has a topology defined on it as follows. First we define two operations $\hull$ and $\ker$: if $X \subseteq \bold{2}^{\Prim(A)}$, then $\ker(X)$ is the intersection of all ideals in $X$, and if $J \subseteq A$ is an ideal, $\hull(J)$ is the set of all primitive ideals in $A$ containing $J$.

\begin{dfn}
A set in $X \subseteq \Prim(A)$ is defined as closed if $X = \hull(\ker(X))$. The topology generated by these closed sets is called the \textit{hull-kernel topology} on $\Prim(A)$.
\end{dfn}

Also, note that by taking the pullback of the hull-kernel topology on $\Prim(A)$ along $\kappa$, we also get a topology on the dual space $\widehat{A}$, which is called the \textit{Fell topology}.

The Fell topology can also be characterised in terms of matrix coefficients of irreducible unitary representations, which we now describe. Given a representation $(\pi, \H_\pi)$ of a C$^*$-algebra $A$, a \textit{positive linear functional associated with $\pi$} is any positive linear functional on $A$ of the form $f(x):= \langle \pi(x) \xi, \xi \rangle$ for $\xi \in \H_\pi$. These linear functionals are also called \textit{diagonal matrix coefficients} of the representation $\pi$. A linear functional is called a \textit{state} if it is normalised i.e.\ has norm 1.

\begin{prop}\cite[Theorem 3.4.10]{Dix77}\label{prop:fell}
Let $A$ be a C$^*$-algebra, $\pi \in \widehat{A}$ and $X \subseteq \widehat{A}$. The following are equivalent:
\begin{enumerate}[(i)]
   \item $\pi \in \overline{X}$; \
   \item At least one of the non-zero positive linear functionals associated with $\pi$ is a weak$^*$-limit of positive linear functionals associated with representations in $X$; \
   \item Every state associated with $\pi$ is a weak$^*$-limit of states associated with representations in $X$.
\end{enumerate}
\end{prop}

We now describe the Fell topology on the group level and state some results about the Fell topology that will be used in this article. We start with the notion of a function of positive type on a group $G$. These are the analogues of positive linear functionals in the C$^*$-algebra context.

\begin{dfn}
Let $G$ be a locally compact group. A \textit{function of positive type} on $G$ is a continuous function $\varphi: G \rightarrow \mathbb{C}$ such that for all $g_1, \dots,g_n \in G$ and $c_1, \dots, c_n \in \mathbb{C}$ 
\begin{displaymath} \sum_{i=1}^n \sum_{j=1}^n c_i\overline{c_j} \varphi(g_ig_j^{-1}) \ge 0. \end{displaymath}
\end{dfn}
It can be checked that every diagonal matrix coefficient corresponding to a unitary representation on a locally compact group is a function of positive type on that group. We thus make the following definition.

\begin{dfn}
Let $G$ be a locally compact group and $(\pi, \H_\pi)$ a unitary representation of $G$. For $\xi \in \H_\pi$, define $\varphi_{\pi,\xi}: G \rightarrow \mathbb{C}, g \mapsto \langle \pi(g) \xi, \xi \rangle$. The functions $\varphi_{\pi,\xi}$ with $\xi$ ranging over all $\xi \in \H_\pi$ are called the \textit{functions of positive type associated with $\pi$}.
\end{dfn}

It should be noted that, by a variation of the GNS construction for groups, one can show that every function of positive type on a locally compact group is a diagonal matrix coefficient of a cyclic unitary representation on that group \cite[Proposition 1.B.10]{BH20}. 

We now define the Fell topology on the unitary dual $\widehat{G}$ of a locally compact group $G$ and then connect convergence in $\widehat{G}$ with the notion of \textit{weak containment} of representations.

\begin{dfn}
Let $G$ be a locally compact group. The \textit{Fell topology} on $\widehat{G}$ is the topology generated by the basis of open sets $W(\pi, \varphi_1, \dots, \varphi_n, \epsilon)$, where $\pi \in \widehat{G}$, $K \subseteq G$ compact, $\varphi_1, \dots, \varphi_n$ functions of positive type associated to $\pi$ and $\epsilon > 0$. The set $W(\pi, K, \varphi_1, \dots, \varphi_n, \epsilon)$ consists of precisely all of those representations $\sigma \in \widehat{G}$ such that for each $\varphi_i$, there exists a function $\psi$ which is a finite sum of functions of positive type associated to $\sigma$ such that 
\begin{displaymath} \lvert \varphi_i(x) - \psi(x) \rvert < \epsilon \end{displaymath}
for all $x \in K$.
\end{dfn}

It can be shown that when $\widehat{G}$ is endowed with this topology, the bijection $\widehat{G} \rightarrow \widehat{C^*(G)}$ is a homeomorphism \cite[Corollary 8.B.5]{BH20}. The Fell topology on $\widehat{C^*(G)}$ can also be defined in terms of the open sets $W(\pi, K, \varphi_1, \dots, \varphi_n, \epsilon)$ and the corresponding matrix coefficients on the C$^*$-algebra level, instead of the definition given earlier in this subsection.

We now briefly discuss \textit{weak containment} of unitary representations which can be used to characterise convergence in the Fell topology. The notion of weak containment of representations is a weakening of the notion of containment of unitary representations. 

\begin{dfn}
Let $G$ be a locally compact group and $(\pi,\H_\pi)$ and $(\sigma, \H_\sigma)$ two unitary representations of $G$. We say the $\pi$ is \textit{weakly contained} in $\sigma$, denoted $\pi \preceq \sigma$, if for every $\epsilon > 0$, every compact $K \subseteq G$ and every $\xi \in \H_\pi$, there exists $\zeta_1, \dots, \zeta_n \in \H_\sigma$ such that 
\begin{displaymath} \lvert \langle \pi(x) \xi, \xi \rangle - \sum_{i=1}^n \langle \sigma(x) \zeta_i, \zeta_i \rangle \rvert < \epsilon \end{displaymath}
for all $x \in K$.
\end{dfn}

It is an important fact that convergence in the Fell topolgy can be characterised in terms of weak containment.

\begin{prop}\cite[Proposition 1.C.8]{BH20}
Let $G$ be a locally compact group and $(\pi_i)_{i\in I}$ a net in $\widehat{G}$. Then the net $(\pi_i)_{i\in I}$ converges to $\pi \in \widehat{G}$ if and only if $\pi \preceq \bigoplus_{j \in J} \pi_j$ for every subnet $(\pi_j)_{j\in J}$ of $(\pi_i)_{i\in I}$.
\end{prop}

Weak containment of representations can also be characterised in terms of kernels of representations on the C$^*$-algebra level. Given a unitary representation $\pi$ of a locally compact group $G$, then recall that $\pi$ extends to a $*$-representation of $C^*(G)$, also denoted by $\pi$. Let $C^*\ker(\pi)$ denote the kernel of $\pi$ when considered as a representation of $C^*(G)$. It is clear that there is a map $\widehat{G} \rightarrow \Prim(C^*(G)), \pi \mapsto C^*\ker(\pi)$. Weak containment can be characterised as follows.

\begin{prop}\cite[Proposition 8.B.4]{BH20}
Let $G$ be a locally compact group and $(\pi,\H_\pi)$ and $(\sigma, \H_\sigma)$ two unitary representations of $G$. Then $\pi \preceq \sigma$ if and only if $C^*\ker(\sigma) \subseteq C^*\ker(\pi)$.
\end{prop}

To finish this subsection on the Fell topology, we state a result about the Fell topology for CCR groups that will be used later. Recall the definition of a CCR C$^*$-algebra in \cite[$\S$4.2]{Dix77} or \cite[pg.\ 270]{Fol95}.

\begin{dfn}
Let $G$ be a locally compact group. The group $G$ is called a \textit{CCR group}, or \textit{CCR} for short, if $C^*(G)$ is CCR.
\end{dfn}

Since every CCR C$^*$-algebra is type I, it follows that every CCR group is type I. We state the following result which provides equivalent characterisations of the CCR property for groups. This is again a consequence of a result of Glimm: see \cite{Gli61} and \cite[$\S$4.2]{Dix77}.

\begin{prop}\label{prop:ccrchar}
Let $G$ be a locally compact second countable group. The following are equivalent:
\begin{enumerate}[(i)]
   \item The group $G$ is CCR; \
   \item For every irreducible unitary representation $\pi$ of $G$ and every $f \in L^1(G)$, $\pi(f)$ is a compact operator; \
   \item The Fell topology on $\widehat{G}$ (equiv.\ $\widehat{C^*(G)}$) is $T_1$. \
\end{enumerate}
\end{prop}


\subsection{Induced representations}

Here we give a quick summary of the definition of an induced representation, in-part to set some notation concerning induced representations, and then we prove a lemma that will be required later in the article.

The reader should refer to the books \cite{Mac76,KT12,Fuh05} for further information and other constructions of induced representations.

Let $G$ be an arbitrary locally compact group and $H$ a closed subgroup of $G$. In the following, $q: G \rightarrow G/H$ will denote the quotient map and $\supp(f)$ will denote the support of a given function $f$. The \textit{modular function} on $G$ is the homomorphism $\Delta_G : \Aut(G) \rightarrow \mathbb{R}_{>0}$ in which $\Delta_G(\alpha)$ ($\alpha \in \Aut(G)$) is defined to be the real number such that $\Delta_G(\alpha)\mu_G(B) = \mu_G(\alpha^{-1}(B))$ for all Borel sets $B \subseteq G$ (see \cite[$\S$15.26]{HR79} for more details on the function $\Delta_G$). If $\alpha$ is the inner-automorphism defined by an element $g \in G$, then we write $\Delta_G(g)$ instead of $\Delta_G(\alpha)$. The modular function on $H$ is similar defined and denoted by $\Delta_H$.

Let $(\pi, \H_\pi)$ be a unitary representation of the subgroup $H$. Define $C_c(G,\pi)$ to be the space of all functions $\xi: G \rightarrow \H_\pi$ satisfying the following properties:
\begin{enumerate}[(i)]
\item $\xi$ is continuous with respect to the norm topology on $\H_\pi$; \
\item $q(\supp(\xi))$ is compact in $G/H$; \
\item For all $x \in G$ and $h \in H$, $\xi(xh) = \Delta_H(h)^{1/2}\Delta_G(h)^{-1/2}\pi(h^{-1}) \xi(x)$. 
\end{enumerate}

We will now discuss how to define an inner-product on $C_c(G,\pi)$. The completion of $C_c(G,\pi)$ with respect to the inner-product that we will define will be the representation space of the induced representation of $\pi$ from $H$ to $G$. 

For a fixed $\varphi \in C_c(G)$, define a function on $G$ by
\begin{displaymath} \varphi^H(x) = \int_H \varphi(xh) \; d\mu_H(h), \; \; (x \in G), \end{displaymath}
where $\mu_H$ denotes the Haar measure on the subgroup $H$. Since $\varphi$ is uniformly continuous, the function $\varphi^H$ is continuous on $G$. Also $\varphi(xh) = \varphi(x)$ for all $x\in G$ and $h \in H$ since $\mu_H$ is left invariant. It follows that there is a well defined continuous function $\varphi^{G/H} \in C_c(G/H)$ given by $\varphi^{G/H}(xH) = \varphi^H(x)$. The next proposition is important in what follows.

\begin{prop}\cite[Propositon 1.9]{KT12}
For every $\psi \in C_c(G/H)$, there exists a $\varphi \in C_c(G)$ such that $\varphi^{G/H} = \psi$.
\end{prop}

Now, for $\xi \in C_c(G,\pi)$, define a function $\lambda_\xi: C_c(G/H) \rightarrow \mathbb{C}$ as follows: given $\psi \in C_c(G/H)$, choose $\varphi \in C_c(G)$ such that $\varphi^{G/H} = \psi$. Then define
\begin{displaymath} \lambda_\xi(\psi) := \int_{G} \varphi(x) \norm{\xi(x)}^2 \; d\mu_G(x) \end{displaymath}
where $d\mu_G$ denotes the Haar measure on $G$. It can be shown that $\lambda_\xi$ is a well-defined positive linear functional on $C_c(G/H)$. Consequently, there exists a positive Radon measure $\mu_\xi$ on $G/H$ such that
\begin{displaymath} \int_G \varphi(x)\norm{\xi(x)}^2 \; d\mu_G(x) = \int_{G/H} \varphi^{G/H}(xH) \; d\mu_\xi(xH). \end{displaymath}

Then, for any $\xi, \zeta \in C_c(G,\pi)$, one can define a complex Radon measure on $G/H$ by 
\begin{displaymath} \mu_{\xi,\zeta} := \frac{1}{4}(\mu_{\xi + \zeta} - \mu_{\xi - \zeta} + i\mu_{\xi + i \zeta} - i\mu_{\xi- i\zeta}). \end{displaymath}
It can be checked that the map $C_c(G, \pi) \times C_c(G, \pi) \rightarrow \mathbb{C}, (\xi, \zeta) \mapsto \mu_{\xi,\zeta}(G/H)$ is a non-degenerate bilinear form on $C_c(G, \pi)$ which makes $C_c(G,\pi)$ into a pre-Hilbert space. The Hilbert space completion of this space is denoted by $\H_{\ind_H^G\pi}$.

The induced representation of $\pi$ from $H$ to $G$ is denoted by $\ind_H^G\pi$ and acts by left translation on functions in $C_c(G, \pi)$ i.e.\ for $x,y \in G$ and $\xi \in C_c(G, \pi)$, $\ind_H^G\pi(x)\xi(y) = \xi(x^{-1}y)$. It can be checked that this representation acts continuously on $C_c(G, \pi)$ and extends to a unitary representations on $\H_{\ind_H^G(\pi)}$. See \cite[$\S$2.3]{KT12} for more information on the above construction.

We now prove the lemma that will be used later in the article. This result is probably already known to experts in the area, but we could not find a proof of this result in the literature in the generality given here.

\begin{lem}\label{lem:autoind}
Let $G$ be a locally compact group, $H$ a closed subgroup of $G$ and $\pi$ a unitary representation of $H$. For any $\alpha \in \Aut(G)$, $\alpha \cdot \ind_H^G\pi \simeq \ind_{\alpha^{-1}(H)}^G \alpha \cdot \pi$, where $\alpha \cdot \pi$ denotes the unitary representation of $\alpha^{-1}(H)$ defined by $\alpha \cdot \pi := \pi \circ \alpha$.
\end{lem}

\begin{proof}
We construct an isometric isomorphism $T: \H_{\ind_{\alpha^{-1}(H)}^G\alpha \cdot \pi} \rightarrow \H_{\ind_H^G \pi}$ which intertwines $\alpha \cdot \ind_H^G\pi$ and $\ind_{\alpha^{-1}(H)}^G \alpha \cdot \pi$. To do this, define a map $T: C_c(G,\alpha \cdot \pi) \rightarrow C_c(G,\pi)$ by $T\xi(x) = \Delta_G(\alpha)^{1/2}\xi(\alpha^{-1}(x))$ ($\xi \in C_c(G,\alpha \cdot \pi), x \in G$). Clearly $T$ is a linear mapping.

The proof now proceeds in two steps. First, we show that $T$ is an isometric isomorphism of the pre-Hilbert space $C_c(G,\alpha \cdot \pi)$ onto $C_c(G,\pi)$ and hence extends to an isometric isomorphism $T: \H_{\ind_{\alpha^{-1}(H)}^G\alpha \cdot \pi} \rightarrow \H_{\ind_H^G \pi}$. Then, we show that $T$ intertwines $\alpha \cdot \ind_H^G\pi$ and $\ind_{\alpha^{-1}(H)}^G \alpha \cdot \pi$.

We may assume that the Haar measures on $H$ and $\alpha^{-1}(H)$ are related by the equation
\begin{displaymath}  \mu_H(B) = \mu_{\alpha^{-1}(H)}(\alpha^{-1}(B)) \end{displaymath}
for any Borel set $B \subseteq H$. Also, we note that $\Delta_H(h) = \Delta_{\alpha^{-1}(H)}(\alpha^{-1}(h))$ for any $h \in H$.

Let $\xi \in C_c(G,\alpha \cdot \pi)$. Then, since $\xi$ is compactly supported modulo $\alpha^{-1}(H)$, there exists a compact subset $K \subseteq G$ such that $\supp(\xi) \subseteq K\alpha^{-1}(H)$. It then follows that $\supp(T\xi) \subseteq \alpha(K)H$, and since $\alpha(K)$ is compact by continuity of $\alpha$, it is thus the case that $T\xi$ is compactly supported modulo $H$. 

For $x \in G$ and $h \in H$, we also have that 
\begin{align*}
T\xi(xh) &= \Delta_G(\alpha)^{1/2}\xi(\alpha^{-1}(xh)) = \Delta_G(\alpha)^{1/2}\xi(\alpha^{-1}(x)\alpha^{-1}(h)) \\
&= \Delta_G(\alpha)^{1/2}\Delta_{\alpha^{-1}(H)}(\alpha^{-1}(h))^{1/2} \Delta_G(\alpha^{-1}(h))^{-1/2} \alpha \cdot \pi(\alpha^{-1}(h^{-1})) \xi(\alpha^{-1}(x)) \\
&= \Delta_{H}(h)^{1/2} \Delta_G(h)^{-1/2}\Delta_G(\alpha)^{1/2} \pi(h^{-1}) \xi(\alpha^{-1}(x)) \\
&= \Delta_{H}(h)^{1/2} \Delta_G(h)^{-1/2} \pi(h^{-1}) T\xi(x). \\
\end{align*}
It is clear that $T\xi$ is continuous since $\alpha$ and $\xi$ are continuous. It thus follows by the above arguments that $T\xi \in C_c(G,\pi)$ for any $\xi \in C_c(G,\alpha \cdot \pi)$. 

To see that $T$ is surjective, suppose that $\eta \in C_c(G,\pi)$. Then, define a function $\xi$ on $G$ by $\xi(x) := \Delta_G(\alpha)^{-1/2} \eta(\alpha(x))$ for $x \in G$. By replacing $\alpha$ with $\alpha^{-1}$ in the above arguments, one checks that $\xi \in C_c(G,\alpha \cdot \pi)$, and it is clear that $T\xi = \eta$, which proves surjectivity of $T$. 

We now show that $T$ is norm preserving. Let $\xi \in C_c(G,\alpha \cdot \pi)$ and let $K \subseteq G$ be a compact set such that $\supp(\xi)$ is contained in $K\alpha^{-1}(H)$. Note that, as seen above, this also implies that $\supp(T\xi)$ is contained in $\alpha(K)H$. Then, choose $\psi \in C_c(G)$ such that for all $x \in \alpha(K)$, $\int_{H} \psi(xt) \; d\mu_{H}(t) = 1$. We then have, by definition of the inner-product on $\H_{\ind_H^G\pi}$,
\begin{align*}
\norm{T\xi}^2 &= \int_G \psi(x) \norm{T\xi(x)}^2 d\mu_G(x) \\
&= \int_G \psi(x) \norm{\Delta_G(\alpha)^{1/2}\xi(\alpha^{-1}(x))}^2 d\mu_G(x) \\
&= \int_G \psi(x) \norm{\xi(\alpha^{-1}(x))}^2 \Delta_G(\alpha)\:d\mu_G(x) \\
&= \int_G \psi(\alpha(x)) \norm{\xi(x)}^2 d\mu_G(x). \\
\end{align*}
Note that the function $\alpha \cdot \psi$ defined by $\alpha \cdot \psi(x) := \psi(\alpha(x))$ ($x \in G$) is in $C_c(G)$ and satisfies $\int_{\alpha^{-1}(H)} \alpha \cdot \psi(xt) \; d\mu_{\alpha^{-1}(H)}(t) = 1$ for all $x \in K$. It then follows by the definition of the inner-product on $\H_{\ind_{\alpha^{-1}(H)}^G \alpha \cdot \pi}$ that
\begin{displaymath} \int_G \psi(\alpha(x)) \norm{\xi(x)}^2 d\mu_G(x) = \norm{\xi}^2 \end{displaymath}

which implies $\norm{\T\xi} = \norm{\xi}$ for all $\xi \in C_c(G,\alpha \cdot \pi)$. Since $C_c(G,\alpha \cdot \pi)$ is dense in $\H_{\ind_{\alpha^{-1}(H)}^G\alpha \cdot \pi}$ and $C_c(G,\pi)$ is dense in $\H_{\ind_H^G \pi}$, it thus follows from the prior arguments that $T$ extends to an isometric isomorphism $T: \H_{\ind_{\alpha^{-1}(H)}^G\alpha \cdot \pi} \rightarrow \H_{\ind_H^G \pi}$.

We now show that $T$ intertwines $\alpha \cdot \ind_H^G \pi$ and $\ind_{\alpha^{-1}(H)}^G \alpha \cdot \pi$. For $\xi \in C_c(G,\alpha \cdot \pi)$ and $x,y \in G$, we have
\begin{align*}
(((\alpha \cdot \ind_H^G\pi)(x))(\T\xi))(y) &=  ((\ind_H^G\pi(\alpha(x))(\T\xi))(y) \\ 
&= T\xi(\alpha(x)^{-1}y) \\
&= T\xi(\alpha(x^{-1})y) \\
&= \Delta(\alpha)^{1/2}\xi(\alpha^{-1}(\alpha(x^{-1})y)) \\
&= \Delta(\alpha)^{1/2}\xi(x^{-1}\alpha^{-1}(y)) \\
&= \Delta(\alpha)^{1/2}((\ind_{\alpha^{-1}(H)}^G \alpha \cdot \pi)(x)) (\xi(\alpha^{-1}(y))) \\
&= T(((\ind_{\alpha^{-1}(H)}^G\alpha \cdot \pi)(x))\xi)(y). \\
\end{align*}

Then, again, since $C_c(G,\alpha \cdot \pi)$ is dense in $\H_{\ind_{\alpha^{-1}(H)}^G\alpha \cdot \pi}$, $C_c(G,\pi)$ is dense in $\H_{\ind_H^G \pi}$, and $T$ is an isometric isomorphism, it follows from the above calculation that $T$ intertwines $\alpha \cdot \ind_H^G \pi$ and $\ind_{\alpha^{-1}(H)}^G \alpha \cdot \pi$. Thus we have that $\alpha \cdot \ind_H^G \pi \simeq \ind_{\alpha^{-1}(H)}^G \alpha \cdot \pi$.
\end{proof}


\subsection{Mackey little group method}

The main mathematical tool used in our proofs is the \textit{Mackey little group method} and its generalisations to crossed product C$^*$-algebras. In this subsection we collect the information that we will need regarding the Mackey little group method on the group level and in the next subsection we collect the information concerning the Mackey little group method in the context of crossed-product C$^*$-algebras. We refer the reader to \cite{Ros94} for a historical overview of the Mackey little group method with details concerning its relation to C$^*$-algebra theory. 

A Borel space $X$ is called \textit{countably separated} if there exists a countable collection $\{ U_i \}_{i=1}^\infty$ of Borel subsets of $X$ such that for any $x \in X$, the intersection of all the $U_i$ containing $x$ is precisely the set $\{x\}$.

Whenever we have a group $G$ and a normal subgroup $N$ of $G$, $G$ acts on $\widehat{N}$ as follows. An element $g \in G$ acts on a unitary representation $\pi$ of $N$ by $g \cdot \pi(n) = \pi(gng^{-1})$. This action then gives rise to an action of $G$ on $\widehat{N}$ by $g \cdot [\pi] = [g\cdot \pi]$ for $\pi$ an irreducible unitary representation of $N$ and $[\pi]$ the corresponding unitary equivalence class. 

The subgroup of all elements in $G$ that stabilise either a representation $\pi$ of $N$ or an equivalence class $[\pi] \in \widehat{N}$ are both denoted by $G_\pi$. Note that $G_\pi$ always contains $N$.

\begin{dfn}
Let $G$ be a locally compact group and $N$ a normal subgroup of $G$. The group $N$ is said to be \textit{regularly embedded} in $G$ if the quotient space $\widehat{N}/G$ is a countably separated Borel space.
\end{dfn}

We give a simple condition on the groups $G$ and $N$ which implies that $N$ is regularly embedded.

\begin{prop}\cite[Theorem 1]{Gli61a}\label{prop:regemb}
Let $G$ be a locally compact second countable group and $N$ a type $\I$ closed normal subgroup of $G$. Then $N$ is regularly embedded in $G$ if and only if the orbits of the action $G \acts \widehat{N}$ are locally closed.
\end{prop}

For use throughout this subsection, we state the definition of quasi-equivalence of unitary representations.

\begin{dfn}
Let $(\pi,\H_\pi)$ and $(\sigma,\H_\sigma)$ be unitary representations of a locally compact group $G$. The representations $(\pi, \H_\pi)$ and $(\sigma, \H_\sigma)$ are \textit{quasi-equivalent} if every non-trivial subrepresentation $\pi'$ of $\pi$ is not disjoint from $\sigma$, and every non-trivial subrepresentation $\sigma'$ of $\sigma$ is not disjoint from $\pi$.
\end{dfn}

\begin{rem}
\begin{enumerate}[(i)]
   \item Two unitary representations $\pi_1$ and $\pi_2$ are quasi-equivalent if and only if there exists a cardinal $n$ such that $n\pi_1$ and $n\pi_2$ are unitary equivalent \cite[Proposition 6.A.4]{BH20}. \
   \item Two quasi-equivalent unitary representations generate isomorphic von Neumann algebras \cite[$\S$5.3.1]{Dix77}.
\end{enumerate}
\end{rem}

We now move onto describing the Mackey little group method. First we describe the Mackey little group method for ordinary unitary representations, and then later we describe it for when multiplier representations are necessary.

In Theorem \ref{thm:Mac1}, we use the term \textit{quasi-orbit} of a representation, and the notion of a quasi-orbit being \textit{concentrated} in an orbit. It is not necessary for the purpose of this article to understand these terms, however, one can consult their definitions in \cite[pg.\ 186]{Mac76}. The important fact that one must know is that if $G$ is a locally compact separable group and $N$ a closed normal type I regularly embedded subgroup of $G$, then the quasi-orbit of every factor representation of $G$ is concentrated in a single orbit. Thus, in this case, every factor representation of $G$, and in particular, every irreducible representation of $G$, can be constructed from a representation of one of the groups $G_\pi$ ($\pi \in \widehat{N}$) via the method described in Theorem \ref{thm:Mac1}. We state this fact as a proposition below due to its importance.

\begin{prop}\cite[pg.\ 186]{Mac76}\label{prop:quasiorbit}
If $N$ is a type I regularly embedded subgroup of $G$, then the quasi-orbit of every factor representation is concentrated in a single orbit of the action $G \acts \widehat{N}$. 
\end{prop}

\begin{thm}\cite[Theorem 3.11]{Mac76}\label{thm:Mac1}
Let $G$ be a separable locally compact group and $N$ a type I closed normal subgroup of $G$. Fix an orbit $\mathcal{O}$ of the action $G \acts \widehat{N}$ and let $\pi$ be a representative of this orbit. Then the map $\pi'' \rightarrow \ind_{G_\pi}^G \pi'' =: \sigma$ establishes, up to unitary equivalence, a one-to-one correspondence between:
\begin{enumerate}[(a)]
   \item the set of all factor representations $\pi''$ of $G_\pi$ such that $\pi''|_N$ is quasi-equivalent to $\pi$; and, \
   \item the set of all factor representations $\sigma$ of $G$ whose quasi-orbits are concentrated in $\mathcal{O}$.
\end{enumerate}
Furthermore, under this correspondence, $\Hom(\pi'',\pi'') \cong \Hom(\sigma, \sigma)$, in particular, the irreducible unitary representations in (a) are in one-to-one correspondence with the irreducible unitary representations in (b).
\end{thm}

\begin{rem}
Note that since $\Hom(\pi'',\pi'') \cong \Hom(\sigma,\sigma)$, it follows that the von Neumann algebras generated by $\pi''$ and $\sigma$ are isomorphic and hence these representations have the same type.
\end{rem}

One may now ask how to compute all the factor representations in part (a) of the theorem. We now describe a general method for computing these factor representations. We first need to define inner and outer Kronecker products of unitary representations.

In the following definition, $\H_\pi \otimes \H_\sigma$ denotes a Hilbert space tensor product of the Hilbert spaces $\H_\pi$ and $\H_\sigma$ and $\pi(x) \otimes \sigma(y)$ is the tensor product of bounded operators on a Hilbert space. See \cite[$\S$2.3]{Dix81} for details on Hilbert space tensor products.

\begin{dfn}
Let $G$ be a locally compact group and $(\pi, \H_\pi)$ and $(\sigma, \H_\sigma)$ be two unitary representations of $G$. The \textit{outer Kronecker product} of $\pi$ and $\sigma$, denoted $\pi \times \sigma$, is the representation of $G \times G$ on the Hilbert space $\H_\pi \otimes \H_\sigma$ defined by $\pi \times \sigma (x,y) := \pi(x) \otimes \sigma(y)$ for $(x,y) \in G \times G$. The \textit{inner Kronecker product} of $\pi$ and $\sigma$, denoted $\pi \otimes \sigma$, is the representation of the group $G$ defined by the restriction of $\pi \times \sigma$ to the diagonal subgroup of $G \times G$.
\end{dfn}

 It can be checked that the inner and outer Kronecker products of unitary representations are unitary representations of the corresponding groups.

We now state the following proposition.

\begin{prop}\cite[Lemma, pg.\ 191]{Mac76}
Retain the notation as in Theorem \ref{thm:Mac1}. Suppose that $\pi$ can be extended to a unitary representation $\pi'$ of $G_\pi$. For any unitary representation $\rho$ of $G_\pi/N$, let $\tilde{\rho}$ denote the canonical lift of $\rho$ to the group $G_\pi$. Then the map $\rho \rightarrow \pi' \otimes \tilde{\rho} =: \pi''$ is a one-to-one correspondence, up to unitary equivalence, between the factor representations of $G_\pi/N$ and the factor representations in (a) of Theorem \ref{thm:Mac1}. Furthermore, $\Hom(\pi'',\pi'') \cong \Hom(\rho,\rho)$.
\end{prop}

Under the assumption that every irreducible unitary representation $\pi$ of $N$ extends to a unitary representation of $G_\pi$, we can restate Theorem \ref{thm:Mac1} as follows, by replacing the representations in $(a)$ of Theorem \ref{thm:Mac1} with the representations produced in the previous proposition.

\begin{thm}\label{thm:Mac2}
Let $G$ be a separable locally compact group and $N$ a type I closed normal subgroup of $G$. Fix an orbit $\mathcal{O}$ of the action $G \acts \widehat{N}$ and let $\pi$ be a representative of this orbit. Suppose that $\pi$ can be extended to a unitary representation $\pi'$ of $G_\pi$. Then the map $\rho \rightarrow \ind_{G_\pi}^G \pi' \otimes \tilde{\rho} =: \sigma$ establishes, up to unitary equivalence, a one-to-one correspondence between:
\begin{enumerate}[(a)]
   \item the set of all factor representations $\rho$ of $G_\pi/N$; and, \
   \item the set of all factor representations $\sigma$ of $G$ whose quasi-orbits are concentrated in $\mathcal{O}$.
\end{enumerate}
Furthermore, under this correspondence, $\Hom(\rho,\rho) \cong \Hom(\sigma, \sigma)$, in particular, the irreducible unitary representations in (a) are in one-to-one correspondence with the irreducible unitary representations in (b).
\end{thm}

A consequence of Theorem \ref{thm:Mac1} and Theorem \ref{thm:Mac2} is the following.

\begin{prop}\label{prop:Mac1}
Let $G$ be a separable locally compact group, $N$ a type I regularly embedded closed normal subgroup of $G$ and $H := G/N$. The following hold:
\begin{enumerate}[(i)]
   \item If for every $\pi \in \widehat{N}$, $G_\pi$ is type I, then $G$ is type I; \
   \item If for every $\pi \in \widehat{N}$, $H_\pi$ is type I, and $\pi$ extends to a unitary representation $\pi'$ of $G_\pi$, then $G$ is type I.
\end{enumerate}
\end{prop}

\begin{proof}
To show that $G$ is type I, it suffices to show that all the factor representations of $G$ are type I \cite[Theorem 6.D.4]{BH20}. Let $\sigma$ be an arbitrary factor representation of $G$. Since $N$ is type I and regularly embedded in $G$ by assumption, it follows by Proposition \ref{prop:quasiorbit} that there exists an orbit $\mathcal{O}$ of the action $G \acts \widehat{N}$ such that the quasi-orbit of $\sigma$ is concentrated in $\mathcal{O}$. Let $\pi$ be a representative of the orbit $\mathcal{O}$. We fix this $\sigma$ and $\pi$ throughout the proofs of (i) and (ii).

Then, to prove (i), suppose that the hypothesis of (i) holds. By Theorem \ref{thm:Mac1}, $\sigma$ is unitary equivalent to a representation of the form $\ind_{G_\pi}^G \pi''$, where $\pi''$ is a factor representation of $G_\pi$ such that $\pi''|_N$ is quasi-equivalent to $\pi$, and $\pi''$ and $\sigma$ have the same type. Since $G_\pi$ is type I, by the hypothesis of (i), $\pi''$ must be type I, and then this implies $\sigma$ is type I. Since $\sigma$ was an arbitrary factor representation of $G$, this proves (i).


To prove (ii), we now suppose the hypotheses of (ii) hold. By Theorem \ref{thm:Mac2}, $\sigma$ is unitary equivalent to a unitary representation of the form $\ind_{G_\pi}^G \pi' \otimes \tilde{\rho}$, where $\pi'$ is an extension of $\pi$ to a unitary representation of $G_\pi$, $\rho$ a factor representation of $H_\pi$, and $\rho$ and $\sigma$ have the same type. Since $H_\pi$ is type I, by the hypotheses of (ii), $\rho$ must be type I, and hence $\sigma$ must be type I. Since $\sigma$ was arbitrary factor representation of $G$, this proves (ii). 
\end{proof}

We now state the more general Mackey little group method in the context of multiplier representations. We note that the definitions of Kronecker products of unitary representations apply equally well to multiplier representations. Also, the inner Kronecker product of a $\omega_1$-representation with a $\omega_2$-representation is a $\omega_{1} \omega_{2}$-representation \cite[$\S$1]{Mac58}.

\begin{thm}\cite[Theorem 3.13]{Mac76}\label{thm:Mac3}
Let $G$ be a separable locally compact group, $N$ a closed normal subgroup of $G$, and $\omega$ a multiplier of $G$ such that $\omega|_{N \times N}$ is a type I multiplier on $N$. Fix an orbit $\mathcal{O}$ of the action of $G \acts \widehat{N}^\omega$ and let $\pi$ be a representative of this orbit. Then, there exists multipliers $\tau_1$ and $\tau_2$ of the groups $G_\pi$ and $G_\pi/N$ respectively, and a $\tau_1$-representation $\pi'$ of $G_\pi$ such that the map $\rho \rightarrow \ind_{G_\pi}^G \pi' \otimes \tilde{\rho} =: \sigma$ establishes, up to unitary equivalence, a one-to-one correspondence between:
\begin{enumerate}[(a)]
   \item the set of all factor $\tau_2$-representations $\rho$ of $G_\pi/N$; and,\
   \item the set of all factor $\omega$-representations $\sigma$ of $G$ whose quasi-orbits are concentrated in $\mathcal{O}$.
\end{enumerate}
Furthermore, under this correspondence, $\Hom(\rho,\rho) \cong \Hom(\sigma, \sigma)$, in particular, the irreducible $\tau_2$-representations in (a) are in one-to-one correspondence with the irreducible $\omega$-representations in (b).
\end{thm}

Remark that if you need to compute regular unitary representations of $G$ using multiplier representations of subquotients of $G$ (i.e.\ the case when $\omega$ is the trivial multiplier in the above theorem), then you take $\tau_2 = \tau_1^{-1}$ in the above theorem.


\subsection{Crossed-product C$^*$-algebras}

To conclude the preliminaries section of the article we introduce crossed product C$^*$-algebras and some of the surrounding theory that will be used.

Crossed product C$^*$-algebras provide a suitable framework for developing an analogue of induced representations and the Mackey little group method in the C$^*$-algebraic framework. The theory of crossed product C$^*$-algebras provides an array of techniques for studying the group C$^*$-algebras of semi-direct products of groups which is useful in the context of this article. A standard reference on the topic that we use here is \cite{Wil07}.

We begin with the definition of a C$^*$-dynamical system.

\begin{dfn}
A \textit{C$^*$-dynamical system} is a triple $(A,G,\alpha)$, where $A$ is a C$^*$-algebra, $G$ is a locally compact group and $\alpha: G \rightarrow \Aut(A), g \mapsto \alpha_g$ is a group homomorphism such that the map $g \mapsto \alpha_g(a)$ is continuous for every $a \in A$. The dynamical system is called \textit{separable} if $A$ is a separable and $G$ is second countable. 
\end{dfn}

We now give two important examples of C$^*$-dynamical systems that will be used in this article.

\begin{exm}\label{exm:dynsys}
(i) Let $X$ be a locally compact Hausdorff space and $G$ a locally compact group acting continuously on $X$. Then, define $\alpha: G \rightarrow \Aut(C_0(X)), g \mapsto \alpha_g$ where $\alpha_g$ is given by the formula
\begin{displaymath} \alpha_g(f)(x) = f(g^{-1} \cdot x) \; \; \; (g \in G, f \in C_0(X), x \in X). \end{displaymath}
The triple $(C_0(X), G, \alpha)$ is then a C$^*$-dynamical system and is referred to as a \textit{locally compact transformation group}. An important special case of this is when $G$ is a locally compact abelian group and $X$ is the Pontryagin dual of $G$.  \\
(ii) It can be shown that if $N \rtimes H$ is a semi-direct product of locally compact groups, then this gives rise to a C$^*$-dynamical system of the form $(C^*(N),H,\alpha)$. Exact details of the definition of $\alpha$ can be found, for example, in \cite[Example 2.3.6]{CELY17}. It is not necessary for the purposes of this article to know the definition of $\alpha$.
\end{exm}

We now work towards defining the crossed product C$^*$-algebra arising from a C$^*$-dynamical system. To do this, we need the notion of a covariant representation of a C$^*$-dynamical system.

\begin{dfn}
Let $(A,G,\alpha)$ be a C$^*$-dynamical system. A \textit{covariant representation} of $(A,G,\alpha)$ is a triple $(\pi, U,\H)$, where $\H$ is an arbitrary Hilbert space, $\pi: A \rightarrow \BH$ a $*$-representation of $A$ and $U: G \rightarrow \UH$ a unitary representation of $G$ satisfying
\begin{displaymath} U(g)\pi(a)U(g^{-1}) = \pi(\alpha_g(a)) \; \; \; (g \in G, a \in A). \end{displaymath}
\end{dfn}

Given a C$^*$-dynamical system $(A,G,\alpha)$ and a covariant representation $(\pi,U,\H)$ of this dynamical system, one can extend this representation to a representation of the algebra $C_c(G,A)$ of all continuous compactly-supported $A$-valued functions on $G$ with the convolution product. This extended representation is denoted by $\pi \rtimes U$ and acts on $f \in C_c(G,A)$ by the following equation
\begin{displaymath} ((\pi \rtimes U)(f))\xi = \int_G \pi(f(g))U(g)\xi \: d\mu_G(g)  \end{displaymath}
for any $\xi \in \H$. It can be shown that $\pi \rtimes U$ is an $L^1$-norm decreasing $*$-representation of $C_c(G,A)$ and it is non-degenerate if $\pi$ is non-degenerate. This is a generalisation of how we extend a unitary representation of a group $G$ to a $*$-representation of $L^1(G)$.

Continuing this generalisation, define a norm on $C_c(G,A)$ by 
\begin{displaymath} \norm{f}_{\text{max}} := \sup\{ \norm{(\pi \rtimes U)(f)} : (\pi,U,\H)\text{ is a convariant representation of }(A,G,\alpha) \}. \end{displaymath}

This then leads us to the definition of a crossed product C$^*$-algebra.

\begin{dfn}
Let $(A,G,\alpha)$ be a C$^*$-dynamical system. The \textit{crossed product C$^*$-algebra} corresponding to this dynamical system is denoted by $A \rtimes_\alpha G$ and is defined as the completion of $C_c(G,A)$ with respect to the norm $\norm{\cdot}_{\text{max}}$. 
\end{dfn}

One can check that a crossed product C$^*$-algebra is indeed a C$^*$-algebra. See \cite[Chapter 2]{Wil07} for the full details. 

There are two examples of crossed product C$^*$-algebras that are used in this article; we list them below.

\begin{exm}
(i) Let $X$ be a locally compact Hausdorff space and $G$ a locally compact group acting continuously on $X$. Let $\alpha: G \rightarrow \Aut(C_0(X))$ be the corresponding action of $G$ on $C_0(X)$. Then as mentioned previously, $(C_0(X),G,\alpha)$ is a C$^*$-dynamical system. The corresponding crossed product C$^*$-algebra is denoted by $C_0(X) \rtimes_\alpha G$. \\
(ii) Let $N \rtimes H$ be a semi-direct product of locally compact groups and $(C^*(N),H,\alpha)$ the corresponding C$^*$-dynamical system referred to in Example \ref{exm:dynsys}. Then the corresponding crossed product C$^*$-algebra $C^*(N) \rtimes_\alpha H$ is isomorphic to $C^*(N \rtimes H)$ \cite[Example 2.3.6]{CELY17}.
\end{exm}

In crossed product C$^*$-algebra theory, one can study induced representations and analogues of the Mackey little group method for crossed product C$^*$-algebras. We do not go into this theory here as it is very extensive and the details are not required in this article, but the reader can view \cite{Wil07} or \cite[Chapter 2]{CELY17} for more details on the topic.

To complete this subsection and the preliminaries, we list two results concerning crossed product C$^*$-algebras that will be used throughout the article.

\begin{prop}\cite[Example 3.16]{Wil07}\label{prop:abcrossprod}
Let $G$ be a locally compact group and suppose that $G = N \rtimes H$, where $N$ and $H$ are locally compact abelian subgroups of $G$, and $N$ is normal in $G$. Let $(C_0(\widehat{N}),H,\alpha)$ denote the C$^*$-dynamical system associated to the action of $H$ on the Pontryagin dual $\widehat{N}$ of $N$. Then $C^*(G)$ is isomorphic to $C_0(\widehat{N}) \rtimes_\alpha H$.
\end{prop}


\begin{thm}\cite[Thm 8.39, Thm 8.44]{Wil07}\label{thm:gcrccrcross}
Let $(C_0(X),G,\alpha)$ be a locally compact transformation group and further assume that $G$ is abelian and second countable. The C$^*$-algebra $C_0(X) \rtimes_\alpha G$ is CCR if and only if the orbits of the action $G \acts X$ are closed, and is GCR if and only if $X/G$ is $T_0$.
\end{thm}

%
%
%
%


\section{Abelian contraction groups}

In this section we study characters on locally compact abelian contraction groups and prove some results that will be used later in the article. This section also provides some simpler examples of contraction groups and contractive scale groups that can be used for motivation later in the artcile.

Throughout this section, if $G$ is a locally compact abelian group, then $\widehat{G}$ denotes its Pontryagin dual, which itself is a locally compact abelian group. The Pontryagin dual and unitary dual coincide as topological spaces for locally compact abelian groups.


\subsection{Characters on abelian contraction groups}

In this section we provide an explicit description of the characters on abelian locally compact contraction groups and show that all such groups are self-dual.

We begin by recalling the description of the characters of $\mathbb{Q}_p$. Given any $x \in \mathbb{Q}_p$, there exists for each $i \in \mathbb{Z}$ coefficients $a_i \in \{0,1,\dots,p-1\}$ such that $x = \sum_{i=-\infty}^\infty a_ip^i$ and the set $\{ a_i : i<0, a_i \ne 0 \}$ is finite. The least index $i$ for which $a_i\ne0$ is precisely the \textit{exponential valuation} $\nu(x)$ of $x$ as defined in \cite[Chapter 2, $\S$2]{Neu99}. Define $[x]:=\sum_{i=0}^\infty a_ip^i$ and 
\begin{displaymath}\{x\}:= \begin{cases} \sum_{i=\nu(x)}^{-1} a_ip^i & \text{if }\nu(x)<0 \\ 0 & \text{if }\nu(x)\ge0 \end{cases} \end{displaymath}
which are the \textit{integer} and \textit{fractional} part of $x$ respectively. Clearly $x = [x] + \{x\}$. Define a map $\psi:\mathbb{Q}_p \rightarrow \mathbb{T}, x \mapsto e^{2\pi i\{x\}}$. Since $\psi$ is identically $1$ on $\mathbb{Z}_p$, it is locally constant and hence continuous. It is then easy to check that $\psi$ is a character on $\mathbb{Q}_p$. For $x,y \in \mathbb{Q}_p$, define $\psi_y(x):=\psi(yx)$. We then have the following result which identifies $\widehat{\mathbb{Q}_p}$ with $\mathbb{Q}_p$.

\begin{prop}\cite[Chap.\ 2, $\S$1.8., Prop.\ 20, Cor.\ (a)]{Bou19}
The map $\mathbb{Q}_p \rightarrow \widehat{\mathbb{Q}_p}, y \mapsto \psi_y$ is an isomorphism of locally compact abelian groups.
\end{prop}

We will now describe the characters on the groups $C_{p^n}(\!(t)\!)$ as defined in Example \ref{exm:abcont}. We define some notation first.

\begin{dfn}
Throughout the remainder of the article, we define $C_{p^n}:= \mathbb{Z}/p^n\mathbb{Z}$. If $x := (x_i)_{i \in \mathbb{Z}} \in C_{p^n}(\!(t)\!)$, we denote by $\nu(x)$ the least integer $k \in \mathbb{Z}$ such that $x_k \ne 0$. We will use $x_i'$ to denote some coset representative of $x_i \in C_{p^n} = \mathbb{Z}/{p^n}\mathbb{Z}$. 
\end{dfn}

For $y := (y_j)_{j \in \mathbb{Z}} \in C_{p^n}(\!(t)\!)$, define a map $\chi_y: C_{p^n}(\!(t)\!) \rightarrow \mathbb{T}$ whose value on an element $x := (x_j)_{j \in \mathbb{Z}} \in C_{p^n}(\!(t)\!)$ is given by
\begin{displaymath} \chi_y(x) = \prod_{j \in \mathbb{Z}} e^{\frac{2\pi i x_j' y_{-j}'}{p^n}}. \end{displaymath}
Note that since only finitely many of the $e^{\frac{2\pi i x_j' y_{-j}'}{p^n}}$ are not 1, it follows that the above product converges to an element of $\mathbb{T}$. For any $y \in C_{p^n}(\!(t)\!)$, $\chi_y$ is identically 1 on the compact open subgroup $K := \{ x \in C_{p^n}(\!(t)\!) : \nu(x) > |\nu(y)| \}$, thus it is locally constant and hence continuous. It can be checked that $\chi_y$ is a well-defined (does not depend on the choice of the coset representatives $x_j'$ and $y_j'$) homomorphism, hence, $\chi_y$ is a character of $C_{p^n}(\!(t)\!)$.

We now prove the following result which shows that all characters of $C_{p^n}(\!(t)\!)$ are of this form and this description of the characters gives rise to an isomorphism of $C_{p^n}(\!(t)\!)$ with its Pontryagin dual. We could not find a reference for this result in the literature, hence we give a complete proof.

\begin{prop}\label{prop:charfpt}
The map $C_{p^n}(\!(t)\!) \rightarrow \widehat{C_{p^n}(\!(t)\!)}, y \mapsto \chi_y$ is an isomorphism of locally compact abelian groups.
\end{prop}

\begin{proof}
Let $\psi$ denote the given map. First we show that $\psi$ is a homomorphism. For all $x,y,z \in C_{p^n}(\!(t)\!)$ we have,
\begin{displaymath} \chi_{y+z}(x) = \prod_{j \in \mathbb{Z}} e^{\frac{2\pi i x_j' (y_{-j}+z_{-j})'}{p^n}} = \prod_{j \in \mathbb{Z}} e^{\frac{2\pi i x_j' y_{-j}'}{p^n}} e^{\frac{2\pi i x_j' z_{-j}'}{p^n}} = \chi_y(x) \chi_z(x).  \end{displaymath}
Thus, $\chi_{y+z} = \chi_y\chi_z$ for all $y,z \in C_{p^n}(\!(t)\!)$. It is also clear that $\chi_0$, where 0 is the identity in $C_{p^n}(\!(t)\!)$, is the trivial character, and for any $y \in C_{p^n}(\!(t)\!)$, $\chi_y^{-1} = \chi_{-y} \in C_{p^n}(\!(t)\!)$ by the previous calculation. Thus $\psi$ is a homomorphism.

Clearly $\psi$ has trivial kernel and hence is injective. To show that $\psi$ is surjective, suppose that $\chi \in \widehat{C_{p^n}(\!(t)\!)}$. We will find $y \in C_{p^n}(\!(t)\!)$ such that $\chi_y = \chi$. For each $j \in \mathbb{Z}$, let $\chi_j$ denote the restriction of $\chi$ to the $j^\text{th}$ factor of the restricted direct product $C_{p^n}(\!(t)\!) = (\bigoplus_{\mathbb{Z}_{<0}} C_{p^n}) \times (\prod_{\mathbb{Z}_{\ge0}} C_{p^n})$. For any $x = (x_j)_{j \in \mathbb{Z}} \in C_{p^n}(\!(t)\!)$, we have that 
\begin{displaymath} \chi(x) = \prod_{j \in \mathbb{Z}} \chi_j(x_j). \end{displaymath}
The restriction $\chi_j$ is a character on the finite cyclic group $C_{p^n}$ for every $j$, so, for every $j$, there exists $y_{-j}' \in C_{p^n}$ such that $\chi_j(z) = e^{\frac{2\pi i z' y_{-j}'}{p^n}}$ for all $z \in C_{p^n}$. Then, we have that 
\begin{displaymath} \chi(x) = \prod_{j \in \mathbb{Z}} e^{\frac{2\pi i x_j' y_{-j}'}{p^n}} \end{displaymath}
for all $x = (x_j)_{j \in \mathbb{Z}} \in C_{p^n}(\!(t)\!)$. To prove surjectivity, it remains to show that $y := (y_j)_{j \in \mathbb{Z}}$, where $y_j := y_j' +p^n\mathbb{Z} \in C_{p^n}$ for each $j$, is an element of $C_{p^n}(\!(t)\!)$. In particular, we need to show that the set $\{ y_j : j<0, y_j \ne 0 \}$ is finite. Since the range of every character on $C_{p^n}(\!(t)\!)$ must be finite, it follows that every character on $C_{p^n}(\!(t)\!)$ must be locally constant. But if the set $\{ y_j : j<0, y_j \ne 0 \}$ is infinite, then this would imply that the character $\chi$ can not be locally constant, which would be a contradiction.

Finally, we show that $\psi$ is a homeomorphism. Note that the sets $U_k := \{ x \in C_{p^n}(\!(t)\!) : \nu(x) \ge k \}$ ($k \in \mathbb{Z}$) form a basis of (compact) neighbourhoods of the identity in $C_{p^n}(\!(t)\!)$. Similarly, the sets $V_{j,k} := \{ \chi \in \widehat{C_{p^n}(\!(t)\!)} : \chi(x) \in B(1,j^{-1}) \text{ for $x \in U_k$} \}$ ($j \in \mathbb{R}_{>1}, k \in \mathbb{Z}$) form a basis of neighbourhoods of the identity in $\widehat{C_{p^n}(\!(t)\!)}$, where $B(x,r)$ denotes the open ball of radius $r$ around $x \in \mathbb{C}$. Now, for $y \in C_{p^n}(\!(t)\!)$ with $\nu(y) = \ell$, and $\chi_y$ its corresponding character, note that $\chi_y(U_k) = \{1\}$ if $k>-\ell$ and $\chi_y(U_k)$ is the $p^n$-th roots of unity if $k \le -\ell$. It follows that $\chi_y \in V_{j,k}$ ($j \in \mathbb{R}_{>1})$ if and only if $y \in U_{-\ell+1}$. This implies $\psi$ is a homeomorphism.
\end{proof}

Recall that the characters of $\mathbb{R}$ are given by $\rho_y(x) := e^{-2\pi i xy}$ $(x,y \in \mathbb{R})$ and that the map $\mathbb{R} \rightarrow \widehat{\mathbb{R}}, y \mapsto \rho_y$ gives an isomorphism of $\mathbb{R}$ with its dual. We consequently obtain the following result.

\begin{prop}\label{prop:selfdualcont}
Every abelian locally compact contraction group $(A,\alpha)$ is self-dual i.e.\ $A$ and $\widehat{A}$ are isomorphic as locally compact abelian groups.
\end{prop}

\begin{proof}
Let $(A, \alpha)$ be an abelian locally compact contraction group. By Theorem \ref{thm:abeliancont}, it follows that $A$ is a direct sum of finitely many copies of groups of the form $\mathbb{R}$, $\mathbb{Q}_p$ and $C_{p^n}(\!(t)\!)$. Since $\mathbb{R}$ is self-dual, and the previous two propositions imply that $\mathbb{Q}_p$ and $C_{p^n}(\!(t)\!)$ are self-dual, the result follows.
\end{proof}

We will now specify a canonical way of identifying a locally compact abelian contraction group $(A,\alpha)$ with its dual that will be used in the following section. So we now fix a locally compact contraction group $(A,\alpha)$. Then, by Theorem \ref{thm:abeliancont}, we have an isomorphism
\begin{displaymath} (A,\alpha) \cong \bigoplus_{p \in \mathbb{P} \cup \{\infty\}} \bigoplus_{f \in \Omega_p} \bigoplus_{n \in \mathbb{N}} (E_{f^n},\alpha_{f^n})^{\mu(p,f,n)} \oplus \bigoplus_{p \in \mathbb{P}} \bigoplus_{n \in \mathbb{N}} (C_{p^n}(\!(t)\!),\alpha_{\text{rs}})^{\nu(p,n)} \end{displaymath}
with notation as defined in the theorem. For the contraction groups $(E_{f^n},\alpha_{f^n})$, we will always, from now on, equip these groups with the basis $B := \{ 1+f^n\mathbb{Q}_p[X], X+f^n\mathbb{Q}_p[X], \dots,X^{\text{deg}(f^n) -1} +f^n\mathbb{Q}_p[X]\}$. Choosing this basis, we get a natural identification of the group $E_{f^n}$ with a direct sum of finitely many copies of $\mathbb{Q}_p$ or $\mathbb{R}$ depending on whether $p$ is finite or $p=\infty$ respectively. Then, the identifications of $\mathbb{Q}_p$ and $\mathbb{R}$ with their duals give rise to a canonical identification of $E_{f^n}$ with its dual. Similarly, Proposition \ref{prop:charfpt} gives a canonical identification of $C_{p^n}(\!(t)\!)$ with its dual. So, by way of these choices, and the above direct sum decomposition of $(A,\alpha)$, we get a canonical identification of $(A,\alpha)$ with its dual.


\subsection[Contractive scale groups with abelian horocycle stabiliser]{Contractive scale groups with abelian horocycle stabiliser}

In this subsection, we show that if $(A,\alpha)$ is an abelian locally compact contraction group, then $A \rtimes_\alpha \mathbb{Z}$ is never CCR. This result will be used in the proof of Theorem \ref{thm:mainthm}. 

\begin{dfn}
Let $(A,\alpha)$ be an abelian locally compact contraction group. The usual action of $\mathbb{Z}$ on $A$ by $n \cdot x = \alpha^n(x)$ ($n \in \mathbb{Z}, x \in A$) we will call the \textit{standard-$\mathbb{Z}$-action} on $A$. The \textit{dual-$\mathbb{Z}$-action} on $A$ is the action of $\mathbb{Z}$ on $A$ obtained from the action of $\mathbb{Z}$ on $\widehat{A}$ after identifying $\widehat{A}$ with $A$ as described at the end of the last section. 
\end{dfn}

We now compute the dual-$\mathbb{Z}$-action on an abelian locally compact contraction group. This is done through the following sequence of lemmas. 

For use in the following lemma, recall the definition of the contraction groups denoted by $(E_{f^n}, \alpha_{f^n})$ from the preliminaries section of this article (see Example \ref{exm:abcont}). The automorphism of $E_{f^n}$ induced by the dual-$\mathbb{Z}$-action will be denoted by $\widehat{\alpha_{f^n}}$.

If $B$ is a basis of the vector space $E_{f^n}$ and $\alpha$ an endomorphism of $E_{f^n}$, then the matrix of $\alpha$ with respect to the basis $B$ is denoted by $\mathcal{M}(\alpha,B)$. We define the \textit{standard basis} of $E_{f^n}$ to be the basis $B := \{ 1+f^n\mathbb{Q}_p[X], X+f^n\mathbb{Q}_p[X], \dots,X^{\text{deg}(f^n) -1} +f^n\mathbb{Q}_p[X]\}$. 

\begin{lem}
Let $p \in \mathbb{P} \cup \{ \infty \}$, $n \in \mathbb{N}$, and $f \in \mathbb{Q}_p[X]$ be a monic irreducible polynomial whose roots have absolute value $<1$ in some algebraic closure of $\mathbb{Q}_p$. Let $B$ be the standard basis of $E_{f^n}$.  Then $\mathcal{M}(\widehat{\alpha_{f^n}},B)$ is precisely the transpose of $\mathcal{M}(\alpha_{f^n},B)$.
\end{lem}

\begin{proof}
Suppose the hypotheses. Suppose that $f^n(X) = X^{m-1} + a_{m-2}X^{m-2} + \cdots + a_0$ with $a_i \in \mathbb{Q}_p$ for each $i$. After equipping $E_{f^n}$ with the standard basis, identify $E_{f^n}$ with a direct sum of $m$ copies of $\mathbb{Q}_p$.

Let $\psi \in \widehat{E_{f^n}}$. Since $E_{f^n}$ is a direct sum of $m$ copies of $\mathbb{Q}_p$, there exists $x_1, \dots, x_m \in \mathbb{Q}_p$ and corresponding characters $\psi_{x_1}, \dots, \psi_{x_m} \in \widehat{\mathbb{Q}_p}$ such that for all $(b_1, \dots, b_m) \in E_{f^n}(=\mathbb{Q}_p^m)$,
\begin{displaymath} \psi(b_1, \dots, b_m) = \psi_{x_1}(b_1) \cdots \psi_{x_m}(b_m). \end{displaymath}
Compute that $\alpha_{f^n}(b_1, \dots, b_m) = (-a_0 b_m, b_1 - a_1b_m, \dots, b_{m-1} - a_{m-1}b_m)$. We now make the following calculation. For $x = -a_0x_1 - a_1x_2 - \cdots - a_{m-1}x_m$, we have that
\begin{align*}
\psi(&\alpha_{f^n}(b_1, \dots, b_m)) = \psi(-a_0 b_m, b_1 - a_1b_m, \dots, b_{m-1} - a_{m-1}b_m) \\
&= \psi_{x_1}(-a_0b_m) \psi_{x_2}(b_1 - a_1b_m) \cdots \psi_{x_m}(b_{m-1} - a_{m-1}b_m) \\
&= \psi_{x_2}(b_1) \psi_{x_3}(b_2) \cdots \psi_{x_m}(b_{m-1}) \psi_{x_1}(-a_0b_m) \psi_{x_2}(-a_1b_m) \cdots \psi_{x_m}(-a_{m-1}b_m) \\
&= \psi_{x_2}(b_1) \psi_{x_3}(b_2) \cdots \psi_{x_m}(b_{m-1}) \psi_{-a_0x_1}(b_m) \psi_{-a_1x_2}(b_m) \cdots \psi_{-a_{m-1}x_m}(b_m) \\
&= \psi_{x_2}(b_1) \psi_{x_3}(b_2) \cdots \psi_{x_m}(b_{m-1}) \psi_{x}(b_m).\
\end{align*}
Note that the second last equality holds regardless of whether the $\psi_{x_i}$ are characters of $\mathbb{Q}_p$ for $p \in \mathbb{P}$ or $p=\infty$. From the above calculation, it follows that 
\begin{displaymath} \widehat{\alpha_{f^n}}(x_1, \dots, x_m) = (x_2, \dots, x_{m}, -a_0x_1 - a_1x_2 - \cdots - a_{m-1}x_m). \end{displaymath}
Then

\begin{displaymath} 
\mathcal{M}(\alpha_{f^n},B) = \begin{pmatrix}
0 & 0 & \cdots & 0 & -a_0 \\
1 & 0 & \cdots & 0 & -a_1 \\
0 & 1 &            &    & -a_2 \\ 
\vdots & \vdots & \ddots & \vdots & \vdots \\
0 & 0 & \cdots & 1 & -a_{m-1} \\
\end{pmatrix} \end{displaymath}

\noindent and

\begin{displaymath} 
\mathcal{M}(\widehat{\alpha_{f^n}},B) = \begin{pmatrix}
0 & 1 & 0 & \cdots & 0  \\
0 & 0 & 1 & \cdots & 0 \\
\vdots & \vdots & \vdots & \ddots & \\
0 & 0 & 0 & \cdots & 1 \\
-a_0 & -a_1 & -a_2 & \cdots   & -a_{m-1} \\
\end{pmatrix} \end{displaymath}
respectively.

The lemma now follows.
\end{proof}

\begin{cor}
Let $p \in \mathbb{P} \cup \{ \infty \}$, $n \in \mathbb{N}$, and $f \in \mathbb{Q}_p[X]$ a monic irreducible polynomial whose roots have absolute value $<1$ in some algebraic closure of $\mathbb{Q}_p$. The automorphism $\widehat{\alpha_{f^n}}$ of $E_{f^n}$ is a contractive automorphism.
\end{cor}

\begin{proof}
If $p = \infty$ (resp.\ $p \in \mathbb{P}$) then $\End(E_{f^n})$ is a finite-dimensional real (resp.\ $p$-adic) vector space. It is a well known fact that all norms on a finite-dimensional real vector space are equivalent, and the same fact also holds for $p$-adic vector spaces \cite[Theorem 5.2.1]{Gou97}. Thus we may assume that $\End(E_{f^n})$ is endowed with the trace norm. Let $B$ be the standard basis of $E_{f^n}$. Since $\mathcal{M}(\widehat{\alpha_{f^n}},B)$ is the transpose of $\mathcal{M}(\alpha_{f^n},B)$ and hence their norms agree in the trace norm, the fact that $\widehat{\alpha_{f^n}}$ is contractive follows directly from the fact that $\alpha_{f^n}$ is contractive.
\end{proof}

Now we describe the dual-$\mathbb{Z}$-action on the groups $C_{p^n}(\!(t)\!)$.

\begin{lem}
The dual-$\mathbb{Z}$-action on $C_{p^n}(\!(t)\!)$ is the left shift.
\end{lem}

\begin{proof}
Let $x,y \in C_{p^n}(\!(t)\!)$ and $n \in \mathbb{Z}$. We compute
\begin{displaymath} n \cdot \chi_y(x) = \chi_y(\alpha_{\text{rs}}^n(x)) = \prod_{j \in \mathbb{Z}} e^{\frac{2\pi i x'_{j-n} y'_{-j}}{p}} = \prod_{j \in \mathbb{Z}} e^{\frac{2\pi i x'_j y'_{-(j+n)}}{p}} = \chi_{\alpha_{\text{rs}}^{-n}(y)}(x). \end{displaymath}
\end{proof}

Since every abelian locally compact contraction group is a finite direct sum of groups of the form $E_{f^n}$ or $C_{p^n}(\!(t)\!)$ (see Theorem \ref{thm:abeliancont}), by combining the previous two lemmas, we get the dual-$\mathbb{Z}$-action on any abelian locally compact contraction group. 

Also, by the previous two lemmas, we get the following result.

\begin{cor}
The automorphism arising from the dual-$\mathbb{Z}$-action on the groups $E_{f^n}$ is a contractive automorphism, and the inverse of the automorphism arising from the dual-$\mathbb{Z}$-action on the groups $C_{p^n}(\!(t)\!)$ is a contractive automorphism. 
\end{cor}

We now prove that if $(A,\alpha)$ is a locally compact abelian contraction group, then the semi-direct product $A \rtimes_\alpha \mathbb{Z}$ is not CCR.

\begin{prop}\label{prop:nonccrabsca}
Let $(A, \alpha)$ be an abelian locally compact contraction group. Then $A \rtimes_\alpha \mathbb{Z}$ is not CCR.
\end{prop}

\begin{proof}
Let $(A,\alpha)$ be an abelian locally compact abelian group. By Proposition \ref{prop:abcrossprod}, we have that $C^*(A \rtimes_\alpha \mathbb{Z})$ is isomorphic to the crossed product C$^*$-algebra $C_0(\widehat{A}) \rtimes_{\widehat{\alpha}} \mathbb{Z}$, where $\widehat{\alpha}: \mathbb{Z} \rightarrow \Aut(C_0(\widehat{A}))$ is the action induced by the action of $\mathbb{Z}$ on $\widehat{A}$. Since $A$ and $\widehat{A}$ are isomorphic by Proposition \ref{prop:selfdualcont}, we have that $C_0(\widehat{A}) \rtimes_{\widehat{\alpha}} \mathbb{Z}$ and $C_0(A) \rtimes_{\beta} \mathbb{Z}$ are isomorphic, where $\beta: \mathbb{Z} \rightarrow \Aut(C_0(A))$ is the action induced by the dual-$\mathbb{Z}$-action on $A$.

Then, by Theorem \ref{thm:gcrccrcross}, to show that $A \rtimes_\alpha \mathbb{Z}$ is not CCR, it suffices to show that the dual-$\mathbb{Z}$-action on $A$ has a non-closed orbit. Decompose $(A,\alpha)$ as in Theorem \ref{thm:abeliancont}
\begin{displaymath} (A,\alpha) \cong \bigoplus_{p \in \mathbb{P} \cup \{\infty\}} \bigoplus_{f \in \Omega_p} \bigoplus_{n \in \mathbb{N}} (E_{f^n},\alpha_{f^n})^{\mu(p,f,n)} \oplus \bigoplus_{p \in \mathbb{P}} \bigoplus_{n \in \mathbb{N}} (C_{p^n}(\!(t)\!),\alpha_\text{rs})^{\nu(p,n)}. \end{displaymath}

Let $x$ be a non-trivial element of one of the individual factors (either one of the $E_{f^n}$ or one of the $C_{p^n}(\!(t)\!)$) in the above direct sum decomposition of $A$; call this factor, say, $B$. By the previous corollary either $\alpha$ or its inverse acts contractively on $B$, so 0 must be a limit point of the orbit of $x$. But the orbit of $x$ does not contain 0 since $x$ is non-trivial, so its orbit cannot be closed, and hence the result follows.
\end{proof}

\section{Torsion-free contraction groups }

In this section we provide a proof of Theorem \ref{thm:mainthm} and we use this theorem to completely classify the irreducible unitary representations of contractive scale groups arising from torsion-free locally compact contraction groups.

Recall from the preliminaries sections, Theorem \ref{thm:contgp}, that any torsion-free contraction group can be expressed as a direct product of a connected simply-connected nilpotent real Lie group and finitely many unipotent linear algebraic groups over $p$-adic fields. The classification result of this section thus depends heavily on the Kirillov orbit method, which provides a description of the unitary dual of such groups in terms of `coadjoint orbits'. We provide an introductory exposition to the Kirillov orbit method in this section prior to using this theory.


\subsection{Contractive scale groups with CCR horocycle stabiliser}

We now proceed with the proof of Theorem \ref{thm:mainthm}. We prove this theorem via a sequence of lemmas. The notation from the statement of the theorem is maintained throughout this section i.e.\ $(N,\alpha)$ is assumed to be a CCR locally compact contraction group and $G := N \rtimes_\alpha \mathbb{Z}$. Also, in the following, given a unitary representation $\pi$ of $N$ and an integer $n$, $n \cdot \pi$ denotes the action of $n$ on $\pi$.

First we note that the automorphism $\alpha$ is \textit{compactly contractive} in the following sense: for each compact subset $K \subseteq N$ and identity neighbourhood $U \subseteq N$, there exists $k \in \mathbb{N}$ such that $\alpha^n(K) \subseteq U$ for all $n \ge k$. See \cite[Lemma 1.4(iv)]{Sie86} for the proof of this result.

\begin{lem}\label{lem:lem41}
Let $\tau$ be the trivial representation of $N$ and let $\pi \in \widehat{N}$ be arbitrary. Then $\overline{\mathbb{Z}(\pi)} = \mathbb{Z}(\pi) \cup \{ \tau \}$. In particular, the orbits of the action $\mathbb{Z} \acts \widehat{N}$ are locally-closed.
\end{lem}

\begin{proof}
For each $n \in \mathbb{Z}_{>0}$, choose a unit vector $\xi_n \in \H_{n \cdot \pi}$. Since $n \cdot \pi$ is irreducible for each $n$, $\xi_n$ must be a cyclic vector for $n \cdot \pi$, for each $n \in \mathbb{Z}_{>0}$. The functions $\varphi_n(x) := \langle n \cdot \pi(x) \xi_n, \xi_n \rangle$ are hence normalised functions of positive type associated with the set of representations $\{ n \cdot \pi \}_{n \in \mathbb{Z}_{>0}}$. Now suppose that there exists $\sigma \in \widehat{N}$ such that $n \cdot \pi \rightarrow \sigma$ in $\widehat{N}$ as $n \rightarrow \infty$. Let $\psi$ be an arbitrary normalised function of positive type associated to $\sigma$. By \cite[Lemma 1.C.6]{BH20}, there exists a suitable choice of the $\xi_n \in \H_{n \cdot \pi}$ so that $\varphi_n \rightarrow \psi$ uniformly on compact sets i.e.\ for every $\epsilon > 0$, and every compact subset $K \subseteq N$, there exists a natural number $M$ such that $n >M$ implies $\norm{\varphi_n(x) - \psi(x)} \le \epsilon$ for all $x \in K$.

Since $\alpha$ is compactly contractive, given $x \in K$ arbitrary, by making $n$ large enough, we can make $\alpha^n(x)$ arbitrarily close to the identity of $N$, and so we can make $\varphi_n(x)$ arbitrarily close to 1. In particular, this implies that $\psi(x)$ must be 1. Since this holds for any compact subset $K \subseteq N$ and any $x \in K$, and since $N$ is a directed union of compact open subgroups, it follows that $\psi$ must be trivial. But this also holds for any function $\psi$ which is a normalised function of positive type associated to $\sigma$. It follows that $\sigma$ must be the trivial representation. The above argument also works if we replace $\{ n \cdot \pi \}_{n \in \mathbb{Z}_{>0}}$ with some subsequence $\{ n_i \cdot \pi \} _{i \in \mathbb{N}}$ with $n_i \rightarrow \infty$ as $i \rightarrow \infty$. In particular, the only limit point of the sequence $\{ n \cdot \pi \}_{n \in \mathbb{Z}_{>0}}$ is the trivial representation.

To complete the proof, we need to show that $\{ n \cdot \pi \}_{n \in \mathbb{Z}_{<0}}$ has no limit points. Note that the induced action of $\alpha$ on $\widehat{N}$ is a homeomorphism \cite[Proposition 1.C.11]{BH20}, which we denote by $\widehat{\alpha}$, and by the previous arguments, $\widehat{\alpha}$ contracts every element of $\widehat{N}$ to the trivial representation. In particular, this implies that for any non-trivial $\pi \in \widehat{N}$, $\widehat{\alpha}^{-n}(\pi)$ must diverge as $n \rightarrow \infty$. Hence, $\{ n \cdot \pi \}_{n \in \mathbb{Z}_{<0}}$, or any of its subsequences, cannot have a limit point. So this completes the proof that $\overline{\mathbb{Z}(\pi)} = \mathbb{Z}(\pi) \cup \{ \tau \}$. 

We finally show that $\mathbb{Z}(\pi)$ is locally-closed. Since $N$ is CCR, the unitary dual $\widehat{N}$ is a $T_1$ topological space by Proposition \ref{prop:ccrchar}. This implies that finite sets in $\widehat{N}$ are closed. In particular, $\overline{\mathbb{Z}(\pi)} \setminus \mathbb{Z}(\pi) = \{\tau\}$ is closed, hence $\mathbb{Z}(\pi)$ is open in $\overline{\mathbb{Z}(\pi)}$, and by definition this means that $\mathbb{Z}(\pi)$ is locally-closed.
\end{proof}

\begin{lem}\label{lem:lem45}
If $\pi \in \widehat{N}$ is non-trivial, then $n\cdot \pi$ is not unitary equivalent to $\pi$ for any non-zero $n \in \mathbb{Z}$. Consequently, $G_\pi = N$ for all non-trivial $\pi \in \widehat{N}$.
\end{lem}

\begin{proof}
Suppose to the contrary that there exists a non-trivial $\pi \in \widehat{N}$ and a non-zero $n \in \mathbb{Z}$ such that $n \cdot \pi \simeq \pi$. Then $\mathbb{Z}_\pi$ is a non-trivial subgroup of $\mathbb{Z}$. Let $m$ be the minimal positive integer in $\mathbb{Z}_\pi$. It follows that $\mathbb{Z}(\pi) = \{\pi, 1 \cdot \pi, \dots, (m-1) \cdot \pi \}$. Since $N$ is CCR and hence $\widehat{N}$ is $T_1$ by Proposition \ref{prop:ccrchar}, it follows that finite sets in $\widehat{N}$ are closed, in particular, $\mathbb{Z}(\pi)$ is closed. But since $n \cdot \pi$ is not equivalent to the trivial representation for any $n \in \mathbb{Z}$, this contradicts the fact that the trivial representation is a limit point of $\mathbb{Z}(\pi)$ by the previous lemma. The result follows.
\end{proof}

\hspace{0.1cm} \textsc{Proof of Theorem \ref{thm:mainthm}.} 
Let $(N,\alpha)$ be an arbitrary locally compact contraction group and $G:=N\rtimes_\alpha \mathbb{Z}$. We first show that $G$ is never CCR. It suffices to show that a quotient of $G$ is not CCR. Now, $G$  has a quotient of the form $G_s = N_s \rtimes_{\alpha_s} \mathbb{Z}$, where $(N_s,\alpha_s)$ is a simple locally compact contraction group. 

We claim that $G_s$ is never CCR. If $N_s$ is connected or totally-disconnected and torsion-free, then it is abelian by Theorem \ref{thm:simpcont}, and Proposition \ref{prop:nonccrabsca} implies that $G_s$ is not CCR. So, by Theorem \ref{thm:simpcont}, it just remains to check the case when $N_s$ is totally-disconnected and torsion. In this case, by applying Theorem \ref{thm:simpcont} again, $N_s$ is isomorphic to $(\bigoplus_{\mathbb{Z}_{<0}} F ) \times (\prod_{\mathbb{Z}_{\ge0}} F )$ for some finite simple group $F$. If $F$ is non-abelian, then by a result that we prove later, Proposition \ref{prop:captors}, the group $(\bigoplus_{\mathbb{Z}_{<0}} F ) \times (\prod_{\mathbb{Z}_{\ge0}} F )$ is not type I and so $G_s$ is not type I (and, in particular, not CCR). If $F$ is abelian, then we apply Proposition \ref{prop:nonccrabsca} again to show that $G_s$ is not CCR. This exhausts all possible cases by Theorem \ref{thm:simpcont} and so it follows that $G_s$ is never CCR, and this implies that $G$ is not CCR since $G_s$ is a quotient of $G$.

Now we proceed with the proof of the points $(i)$ and $(ii)$ in the theorem. So now assume that $(N,\alpha)$ is an arbitrary locally compact contraction group and further assume that $N$ is CCR.

$(i)$ We claim that the result follows directly from Proposition \ref{prop:Mac1}(ii). We need to check that all the required assumptions in Proposition \ref{prop:Mac1}(ii) hold in this situation. Obviously $N$ is type I by assumption. By Lemma \ref{lem:lem41}, each orbit of the action $G \acts \widehat{N}$ is locally closed, and it thus follows by Proposition \ref{prop:regemb} that $N$ is regularly embedded in $G$. If $\tau$ is the trivial representation of $N$, then the stabiliser subgroup of $\tau$ under the action of $G$ is $G_\tau = G$, and $\tau$ extends trivially to $G$. If $\pi \in \widehat{N}$ is non-trivial, then Lemma \ref{lem:lem45} shows that $G_\pi = N$, and it is clear that $\pi$ (in the trivial way) extends to a unitary representation of $G_\pi$. So all of the assumptions in Proposition \ref{prop:Mac1}(ii) hold. Then, since $H = G/N \cong \mathbb{Z}$ is abelian and hence all of its subgroups are type I, it follows by Proposition \ref{prop:Mac1}(ii) that $G$ is type I.

$(ii)$ We apply Theorem \ref{thm:Mac2}. Let $X$ be a cross-section of the orbits of the action $G \acts \widehat{N}$. Let $\tau$ be the trivial representation of $N$, which is necessarily contained in $X$. Then $G_{\tau} = G$, so $\tau'$, the extension of $\tau$ to $G$ as in the notation of  Theorem \ref{thm:Mac2}, can be taken to be the trivial representation on $G$. Then the irreducible unitary representations of $G$ in Theorem \ref{thm:Mac2} that correspond to the orbit of $\tau$ are precisely the representations $\ind_G^G(\tau' \otimes \psi) \simeq \tau' \otimes \psi \simeq \psi$ for $\psi \in \widehat{G_\tau/N} \cong \widehat{\mathbb{Z}_{\tau}} = \widehat{\mathbb{Z}}$.

For a non-trivial $\pi \in X$, we have by Lemma \ref{lem:lem45} that $G_\pi = N$.  Then, we may take $\pi'$, as in the notation of Theorem \ref{thm:Mac2}, to be $\pi$ itself. Then, there is only one irreducible unitary representation of $G$ corresponding to the orbit of $\pi$ in Theorem \ref{thm:Mac2}, and it is precisely the representation $\ind_N^G(\pi' \otimes \tau) \simeq \ind_N^G\pi' = \ind_N^G\pi$ since $G_\pi/N = N/N$ is the trivial group ($\tau$ here we take to denote the trivial representation of the trivial group). It thus follows by Theorem \ref{thm:Mac2} that $\widehat{G} = \{ \text{ind}_{N}^G \pi : \pi \in X \} \cup \widehat{\mathbb{Z}}$, because $N$ is regularly embedded in $G$, and so every irreducible unitary representation of $G$ has its quasi-orbit concentrated in one of the orbits of the representations in $X$ by Proposition \ref{prop:quasiorbit}. \qed


\subsection{Kirillov orbit method}

Let $G$ be a connected simply-connected nilpotent real Lie group and let $\Lie(G)$ denote its Lie algebra. 

We use $\Lie(G)^*$ to denote the real vector space dual of $\Lie(G)$ i.e. the vector space of real-valued linear functionals on $\Lie(G)$. The group $G$ acts on $\Lie(G)^*$ via the \textit{coadjoint action} $\Ad^*_G : G \rightarrow \Aut(\Lie(G)^*)$ which is defined by
\begin{displaymath} (\Ad^*_G(g)(\lambda))(X) := \lambda(\Ad_G(g^{-1})(X)), \; \; \; (g \in G, X \in \Lie(G), \lambda \in \Lie(G)^*), \end{displaymath}
where $\Ad_G$ denotes the usual adjoint representation of $G$.

The \textit{coadjoint orbit space}, denoted $\mathcal{CO}(G)$, is the space $\Lie(G)^*/\Ad^*_G(G)$ equipped with the quotient topology.

The following is an important theorem of Kirillov, often referred to as the \textit{Kirillov orbit method}.

\begin{thm} \cite{Kir59,Kir62,Kir04}
Let $G$ be a connected simply-connected nilpotent real Lie group. The following hold:
\begin{enumerate}[(i)]
   \item There is a bijection between $\widehat{G}$ and $\mathcal{CO}(G)$; \
   \item The group $G$ is CCR.
\end{enumerate}
\end{thm}

Later, it was shown that Kirillov's bijection between $\widehat{G}$ and $\mathcal{CO}(G)$ is a homeomorphism, see \cite{Bro73}. It is shown in \cite{Mor65} that the above theorem also holds when $G$ is replaced with a unipotent linear algebraic groups over a $p$-adic field, and the map is shown to be a homeomorphism in the $p$-adic case in \cite{BS08}.

We now give a brief explanation of how one constructs the irreducible unitary representations of the group $G$ using the Kirillov orbit method. Let $\lambda \in \Lie(G)^*$. We say that a subalgebra $\mathfrak{h} \subseteq \Lie(G)$ is \textit{subordinate} to $\lambda$ if $\lambda([x,y]) = 0$ for all $x,y \in \mathfrak{h}$. It is an important fact that such a subalgebra always exists and is occasionally referred to as a \textit{real algebraic polarization} \cite{Kir04}. Now, define a character $\sigma_{\lambda,\mathfrak{h}}$ on $\exp(\mathfrak{h})$ by $\sigma_{\lambda,\mathfrak{h}}(\exp(h)) := e^{i\lambda(h)}$ for $h \in \mathfrak{h}$, and induce it to a unitary representation $\pi_{\lambda,\mathfrak{h}} := \ind_{\exp(\mathfrak{h})}^G \sigma_{\lambda,\mathfrak{h}}$ of $G$.

Then, we have the following theorem.

\begin{thm}\cite[Theorem 5.2]{Kir62}
Let $G$ be a connected simply-connected nilpotent real Lie group. The following hold:
\begin{enumerate}[(i)]
   \item Let $\lambda \in \Lie(G)^*$ and $\mathfrak{h} \subseteq \Lie(G)$ a subalgebra subordinate to $\lambda$. The representation $\pi_{\lambda,\mathfrak{h}}$ is irreducible if and only if $\mathfrak{h}$ is maximal among subalgebras of $\Lie(G)$ subordinate to $\lambda$; \
   \item Let $\lambda_1, \lambda_2 \in \Lie(G)^*$ and suppose that $\mathfrak{h}_1$ and $\mathfrak{h}_2$ are subalgebras of $\Lie(G)$ maximally subordinate to $\lambda_1$ and $\lambda_2$ respectively. Then the representations $\pi_{\lambda_1,\mathfrak{h}_1}$ and $\pi_{\lambda_2, \mathfrak{h}_2}$ are equivalent if and only if $\lambda_1$ and $\lambda_2$ are equal modulo the coadjoint action; \
   \item If $\rho$ is an irreducible unitary representation of $G$, then there exists a $\lambda \in \Lie(G)^*$ and a subalgebra $\mathfrak{h} \subseteq \Lie(G)$ maximally subordinate to $\lambda$ such that $\rho \simeq \pi_{\lambda,\mathfrak{h}}$. \
\end{enumerate}
\end{thm}

Again, in the paper \cite{Mor65}, it is shown that the above theorem holds identically in the case when $G$ is a unipotent linear algebraic group over a $p$-adic field.


\subsection{Classification of the irreducible unitary representations of contractive scale groups with torsion-free horocycle stabiliser}

In this subsection we give a version of the Kirillov orbit method for semi-direct products of the form $N \rtimes_\alpha \mathbb{Z}$ where $(N,\alpha)$ is a torsion-free locally compact contraction group.

Note that if $\mathfrak{n}$ is a Lie algebra and $\eta$ is an endomorphism of $\mathfrak{n}$, then $\eta$ induces an action on $\mathfrak{n}^*$ by $\eta \cdot \lambda (n) = \lambda(\eta(n))$ for $n \in \mathfrak{n}, \lambda \in \mathfrak{n}^*$. If $\mathfrak{n}$ is the Lie algebra of a (possibly $p$-adic) Lie group $N$, then this also induces an action of $\eta$ on $\mathcal{CO}(N)$. 

If $\alpha: G \rightarrow H$ is a homomorphism of Lie groups or $p$-adic unipotent linear algebraic groups, then we denote by $\Lie(\alpha): \Lie(G) \rightarrow \Lie(H)$ the Lie algebra morphism corresponding to taking the differential of $\alpha$ at the identity of $G$. 

Recall that for a connected simply-connected nilpotent real Lie group $N$, the exponential function on $N$, $\exp_N:\Lie(N) \rightarrow N$, is a diffeomorphism. Also, for any endomorphism $\alpha$ of $N$, the following diagram commutes.

\begin{displaymath}
\xymatrix@C=1.4cm@R=1.4cm
			{
			N \ar@[black]@{->}[r]^\alpha \ar@[black]@{<-}[d]_{\exp_N} & N \ar@[black]@{<-}[d]^{\exp_N} \\
		       \Lie(N) \ar@[black]@{->}[r]^{\Lie(\alpha)} & \Lie(N) 
			}
\end{displaymath}

On the other hand, if $k$ is a $p$-adic local field and $N$ the group of rational points of a unipotent linear algebraic group over $k$, then there also exists a globally defined exponential function from $\Lie(N)$ to $N$ satisfying the above commutative diagram. This follows from the fact that, in this case, $N$ is a group of unipotent matrices, and $\Lie(N)$ is a nilpotent Lie algebra, so the series
\begin{displaymath} \exp(X) = \sum_{i=0}^{\infty} \frac{X^n}{n!}  \end{displaymath}
defining the exponential function is a finite sum for any $X \in \Lie(N)$ and hence converges for any such $X$. It is also shown in \cite[Page 159]{Mor65} that this exponential function is an isomorphism of varieties in the case of unipotent linear algebraic groups over the $p$-adics.

The following lemma will be required in the following.

\begin{lem}\label{lem:kiract}
Let $N$ be either a connected simply-connected nilpotent real Lie group or the group of rational points of a unipotent linear algebraic group over a $p$-adic field and $\alpha$ a contractive automorphism of $N$. Then, under the identification of $\widehat{N}$ with $\mathcal{CO}(N)$, the action of $\alpha$ on $\widehat{N}$ is identified with the induced action of $\Lie(\alpha)$ on $\mathcal{CO}(N)$.
\end{lem}

\begin{proof}
Assume the hypotheses of the lemma. Fix $\pi \in \widehat{N}$. By the work of Kirillov, there exists a $\lambda \in \Lie(N)^*$ and a subalgebra $\mathfrak{h} \subseteq \Lie(N)$ maximally subordinate to $\lambda$ such that $\pi \simeq \pi_{\lambda,\mathfrak{h}}$. Then, by Lemma \ref{lem:autoind}, we have that

\begin{align*}
\alpha \cdot \pi &\simeq \alpha \cdot \pi_{\lambda,\mathfrak{h}} \\
&= \alpha \cdot \ind_{\exp(\mathfrak{h})}^N \sigma_{\lambda, \mathfrak{h}} \\
&\simeq \ind_{\alpha^{-1}(\exp{\mathfrak{h})}}^N \alpha \cdot \sigma_{\lambda, \mathfrak{h}} \\
&= \ind_{\exp{(\Lie(\alpha)^{-1}(\mathfrak{h}))}}^N \alpha \cdot \sigma_{\lambda, \mathfrak{h}}.
\end{align*}

Furthermore, for $h \in \Lie(\alpha)^{-1}(\mathfrak{h})$,

\begin{align*}
\alpha \cdot \sigma_{\lambda,\mathfrak{h}}(\exp(h)) &= \sigma_{\lambda,\mathfrak{h}}(\alpha(\exp(h))) \\
&= \sigma_{\lambda,\mathfrak{h}}(\exp(\Lie(\alpha)(h))) \\
&= e^{i \lambda(\Lie(\alpha)(h))} \\
&= e^{i (\Lie(\alpha)\cdot\lambda)(h)} \\
&= \sigma_{\Lie(\alpha)\cdot\lambda, \Lie(\alpha)^{-1}(\mathfrak{h})}(\exp(h))
\end{align*}

from which it follows that 

\begin{displaymath} \alpha \cdot \pi \simeq \alpha \cdot \pi_{\lambda, \mathfrak{h}} \simeq \ind_{\exp(\Lie(\alpha)^{-1}(\mathfrak{h}))}\sigma_{\Lie(\alpha) \cdot \lambda, \Lie(\alpha)^{-1}(\mathfrak{h})} = \pi_{\Lie(\alpha)\cdot\lambda, \Lie(\alpha)^{-1}(\mathfrak{h})}. \end{displaymath} 
\end{proof}


Recall that given a torsion-free locally compact contraction group $(N,\alpha)$, there is an isomorphism 

\begin{displaymath} N \cong N_0 \times N_{p_1} \times \cdots \times N_{p_n} \end{displaymath}

where $N_0$ is the connected component of the identity in $N$, a connected simply-connected nilpotent real Lie group, and the $N_{p_i}$ are unipotent linear algebraic groups over $\mathbb{Q}_{p_i}$ for each $i$. In particular, the Kirillov orbit method applies to each of the factors in the above direct product decomposition of $N$ and so $N$ is a CCR locally compact contraction group. This means that Theorem \ref{thm:mainthm} applies to this particular group $N$. We use these facts in the proof of the following theorem. 

Note that, by Theorem \ref{thm:contgp}, the groups $N_0, N_{p_1}, \dots, N_{p_n}$ are invariant under the contractive automorphism $\alpha$. Let $\alpha_0, \alpha_{p_1}, \dots,\alpha_{p_n}$ denote the restrictions of $\alpha$ to each of the subgroups $N_0, N_{p_1}, \dots, N_{p_n}$ respectively. Then, $\alpha$ acts on the direct product 
\begin{displaymath} N \cong N_0 \times N_{p_1} \times \cdots \times N_{p_n} \end{displaymath}
by the automorphism $\alpha_0 \times \alpha_{p_1} \times \cdots \times \alpha_{p_n}$. This also gives rise to an action of $\mathbb{Z}$ on 
\begin{displaymath} \widehat{N} \cong \widehat{N_0} \times \widehat{N_{p_1}} \times \cdots \times \widehat{N_{p_n}} \end{displaymath}
and a homeomorphism
\begin{displaymath} \widehat{N}/\mathbb{Z} \cong (\widehat{N_0} \times \widehat{N_{p_1}} \times \cdots \times \widehat{N_{p_n}})/\mathbb{Z}. \end{displaymath}
Similarly, $\Lie(\alpha_0) \oplus \Lie(\alpha_{p_1}) \oplus \cdots \oplus \Lie(\alpha_{p_n})$ acts on the Lie algebra
\begin{displaymath} \Lie(N_0) \oplus \Lie(N_{p_1}) \oplus \cdots \oplus \Lie(N_{p_n}) \end{displaymath}
which gives rise to an action of $\mathbb{Z}$ on
\begin{displaymath} \mathcal{CO}(N_0) \times \mathcal{CO}(N_{p_1}) \times \cdots \times \mathcal{CO}(N_{p_n}) \end{displaymath}
and, by Lemma \ref{lem:kiract}, homeomorphisms
\begin{displaymath} (\mathcal{CO}(N_0) \times \mathcal{CO}(N_{p_1}) \times \cdots \times \mathcal{CO}(N_{p_n}))/\mathbb{Z} \cong (\widehat{N_0} \times \widehat{N_{p_1}} \times \cdots \times \widehat{N_{p_n}})/\mathbb{Z} \cong \widehat{N}/\mathbb{Z}. \end{displaymath}

In the following theorem, we use the notation as above, and we understand the actions of $\mathbb{Z}$ on $\widehat{N} \cong \widehat{N_0} \times \widehat{N_{p_1}} \times \cdots \times \widehat{N_{p_n}}$ and $\mathcal{CO}(N_0) \times \mathcal{CO}(N_{p_1}) \times \cdots \times \mathcal{CO}(N_{p_n})$ as the ones described above.

\begin{thm}\label{thm:kirclass}
Let $(N,\alpha)$ be a torsion-free locally compact contraction group and set $G := N \rtimes_\alpha \mathbb{Z}$. Let $N_0$ be the connected component of the identity in $N$ and let $p_1, \dots, p_n$ be distinct primes such that there exists a unipotent linear algebraic group $N_{p_i}$ over $\mathbb{Q}_{p_i}$ for each $i$ satisfying 
\begin{displaymath} N \cong N_0 \times N_{p_1} \times \cdots \times N_{p_n}. \end{displaymath}
The following hold:
\begin{enumerate}[(i)]
   \item Let $\iota_{0}, \iota_{p_1},\dots, \iota_{p_n}$  be the trivial functionals in $\Lie(N_0)^*, \Lie(N_{p_1})^*, \dots, \Lie(N_{p_n})^*$ respectively and let $\mathcal{I}$ denote the orbit of $\iota_{0} \oplus \iota_{p_1} \oplus \cdots \oplus \iota_{p_n}$ in 
\begin{displaymath} (\mathcal{CO}(N_0) \times \mathcal{CO}(N_{p_1}) \times \cdots \times \mathcal{CO}(N_{p_n}))/\mathbb{Z}. \end{displaymath}  
 Then, there exists a bijection between $\widehat{G}$ and the set
   \begin{displaymath}((\mathcal{CO}(N_0) \times \mathcal{CO}(N_{p_1}) \times \cdots \times \mathcal{CO}(N_{p_n}))/\mathbb{Z}) \setminus \{ \mathcal{I} \} \cup \mathbb{T}; \end{displaymath} \
   \item Let $X \subseteq \Lie(N_0)^* \oplus \Lie(N_{p_1})^* \oplus \dots \oplus \Lie(N_{p_n})^*$ be a set of representatives of the orbit space 
\begin{displaymath} ((\mathcal{CO}(N_0) \times \mathcal{CO}(N_{p_1}) \times \dots \times \mathcal{CO}(N_{p_n}))/\mathbb{Z}) \setminus \{\mathcal{I}\}.\end{displaymath} 
For each representative $\lambda = \lambda_0 \oplus \lambda_{p_1} \oplus \cdots \oplus \lambda_{p_n} \in X$, where $\lambda_0 \in \Lie(N_0)^*$ and $\lambda_{p_i} \in \Lie(N_{p_i})^*$ for each $i$, let $\mathfrak{h}_{\lambda_0}, \mathfrak{h}_{\lambda_{p_1}}, \dots, \mathfrak{h}_{\lambda_{p_n}}$ be subalgebras of $\Lie(N_0), \Lie(N_{p_1}), \dots, \Lie(N_{p_n})$ respectively that are maximally subordinate to $\lambda_0, \lambda_{p_1}, \dots, \lambda_{p_n}$ respectively. Define $\pi_{\lambda_0,\lambda_{p_1},\dots,\lambda_{p_n}}$ to be the unitary representation
   \begin{displaymath} \ind_{\exp(\mathfrak{h}_{\lambda_0}) \times \exp(\mathfrak{h}_{\lambda_{p_1}}) \times \cdots \times \exp(\mathfrak{h}_{\lambda_{p_n}})}^G \sigma_{\lambda_0, \mathfrak{h}_{\lambda_0}} \times \sigma_{\lambda_{p_1},\mathfrak{h}_{\lambda_{p_1}}} \times \cdots \times \sigma_{\lambda_{p_n}, \mathfrak{h}_{\lambda_{p_n}}}. \end{displaymath}
   Then $\widehat{G} = \{ \pi_{\lambda_0,\lambda_{p_1},\dots,\lambda_{p_n}} : \lambda = \lambda_0 \oplus \lambda_{p_1} \oplus \cdots \oplus \lambda_{p_n} \in X \} \cup \widehat{\mathbb{Z}}$.
\end{enumerate}
\end{thm}

\begin{proof}
$(i)$ By Theorem \ref{thm:mainthm}, for any cross-section $X$ of the non-trivial orbits of the action $\mathbb{Z} \acts \widehat{N}$, we have that $\widehat{G} = \{ \ind_N^G \pi : \pi \in X \} \cup \widehat{\mathbb{Z}}$. To prove $(i)$, it suffices to show that $X$ can be chosen in bijection with the set 
\begin{displaymath} (\mathcal{CO}(N_0) \times \mathcal{CO}(N_{p_1}) \times \cdots \times \mathcal{CO}(N_{p_n}))/\mathbb{Z} \setminus \{ \mathcal{I} \}. \end{displaymath} 

As discussed prior to the theorem, there are homeomorphisms
\begin{displaymath} \widehat{N}/\mathbb{Z} \cong (\widehat{N_0} \times \widehat{N_{p_1}} \times \cdots \times \widehat{N_{p_n}})/\mathbb{Z} \cong (\mathcal{CO}(N_0) \times \mathcal{CO}(N_{p_1}) \times \cdots \times \mathcal{CO}(N_{p_n}))/\mathbb{Z}. \end{displaymath}
Since the cross-section $X$ must be in bijection with $\widehat{N}/\mathbb{Z}$ minus the trivial orbit, $(i)$ follows immediately.

$(ii)$ Since $\widehat{N} = \widehat{N_0} \times \widehat{N_{p_1}} \times \cdots \times \widehat{N_{p_n}}$, every irreducible representation of $N$ must be of the form 
\begin{displaymath} \ind_{\exp(\mathfrak{h}_{\lambda_0})}^N \sigma_{\lambda_0, \mathfrak{h}_{\lambda_0}} \times \ind_{\exp(\mathfrak{h}_{\lambda_{p_1}})}^N \sigma_{\lambda_{p_1}, \mathfrak{h}_{\lambda_{p_1}}} \times \cdots \times \ind_{\exp(\mathfrak{h}_{\lambda_{p_n}})}^N \sigma_{\lambda_{p_n}, \mathfrak{h}_{\lambda_{p_n}}} \end{displaymath}
where $\lambda_0, \lambda_{p_1}, \dots, \lambda_{p_n}$ are elements of $\Lie(N_0)^*, \Lie(N_{p_1})^*, \dots,\Lie(N_{p_n})^*$ respectively and $\mathfrak{h}_{\lambda_0}, \mathfrak{h}_{\lambda_{p_1}}, \dots, \mathfrak{h}_{\lambda{p_n}}$ are maximally subordinate subalgebras of $\Lie(N_0), \Lie(N_{p_1}), \dots, \Lie(N_{p_n})$ corresponding to $\lambda_0, \lambda_{p_1}, \dots, \lambda_{p_n}$ respectively.
By \cite[Theorem 2.53]{KT12}, this representation is unitary equivalent to the representation 
\begin{displaymath} \ind_{\exp(\mathfrak{h}_{\lambda_0}) \times \exp(\mathfrak{h}_{\lambda_{p_1}}) \times \cdots \times \exp(\mathfrak{h}_{\lambda_{p_n}})}^N \sigma_{\lambda_0, \mathfrak{h}_{\lambda_0}} \times \sigma_{\lambda_{p_1},\mathfrak{h}_{\lambda_{p_1}}} \times \cdots \times \sigma_{\lambda_{p_n}, \mathfrak{h}_{\lambda_{p_n}}}. \end{displaymath}
Thus, all the irreducible representations of $G$ of the form $\ind_N^G\pi$ as in Theorem \ref{thm:mainthm} must be of the form
\begin{displaymath}  \ind_N^G\ind_{\exp(\mathfrak{h}_{\lambda_0}) \times \exp(\mathfrak{h}_{\lambda_{p_1}}) \times \cdots \times \exp(\mathfrak{h}_{\lambda_{p_n}})}^N \sigma_{\lambda_0, \mathfrak{h}_{\lambda_0}} \times \sigma_{\lambda_{p_1},\mathfrak{h}_{\lambda_{p_1}}} \times \cdots \times \sigma_{\lambda_{p_n}, \mathfrak{h}_{\lambda_{p_n}}}, \end{displaymath}
and by induction in stages \cite[Theorem 2.47]{KT12}, such a representation is unitary equivalent to the representation
\begin{displaymath}  \ind_{\exp(\mathfrak{h}_{\lambda_0}) \times \exp(\mathfrak{h}_{\lambda_{p_1}}) \times \cdots \times \exp(\mathfrak{h}_{\lambda_{p_n}})}^G \sigma_{\lambda_0, \mathfrak{h}_{\lambda_0}} \times \sigma_{\lambda_{p_1},\mathfrak{h}_{\lambda_{p_1}}} \times \cdots \times \sigma_{\lambda_{p_n}, \mathfrak{h}_{\lambda_{p_n}}}. \end{displaymath}
Finally, allowing $\lambda = \lambda_0 \oplus \lambda_{p_1} \oplus \cdots \oplus \lambda_{p_n}$ to vary over all $\lambda \in X$ guarantees that the set $\{ \pi_{\lambda_0,\lambda_{p_1},\dots,\lambda_{p_n}} : \lambda = \lambda_0 \oplus \lambda_{p_1} \oplus \cdots \oplus \lambda_{p_n} \in X \} \cup \widehat{\mathbb{Z}}$ is a set of representatives of the equivalence classes of irreducible unitary representations of $G$. Part $(ii)$ then follows.
\end{proof}


\section{Torsion contraction groups}\label{chap:5}

The unitary representation theory of the torsion component $\tor(N)$ of a locally compact contraction group $(N,\alpha)$ is not well understood and there is very limited understanding as to what can happen in the case of the representation theory of these groups. The purpose of this section is to initiate a line of research that will work on the problem of determining the type I torsion locally compact contraction groups and classify their irreducible unitary representations.

Unlike the torsion-free case, there are examples of non-type-I torsion locally compact contraction groups. To the best of the authors knowledge, however, at the time of writing this article, there is only one known class of examples of non-type-I torsion locally compact contraction groups. We present this class of groups in the proposition below and we also provide the proof as it may be of interest to the reader.

\begin{prop}\cite[Proposition 5.8]{CKM21}\label{prop:captors}
Let $F$ be a finite simple non-abelian group. The group $(\bigoplus_{\mathbb{Z}_{<0}}F) \times (\prod_{\mathbb{Z}_{\ge 0}} F)$ equipped with the right-shift automorphism is a torsion locally compact contraction group and is not type $\I$.
\end{prop}

\begin{proof}
It is obvious that the group $(\bigoplus_{\mathbb{Z}_{<0}}F) \times (\prod_{\mathbb{Z}_{\ge 0}} F)$ with the right-shift automorphism is a torsion locally compact contraction group. To show that $(\bigoplus_{\mathbb{Z}_{<0}}F) \times (\prod_{\mathbb{Z}_{\ge 0}} F)$ is not type $\I$, it suffices to show that a quotient of $(\bigoplus_{\mathbb{Z}_{<0}}F) \times (\prod_{\mathbb{Z}_{\ge 0}} F)$ is not type $\I$. The group $\bigoplus_{\mathbb{Z}_{<0}}F$ is a quotient of $(\bigoplus_{\mathbb{Z}_{<0}}F) \times (\prod_{\mathbb{Z}_{\ge 0}} F)$ and it is discrete. By Thoma's theorem \cite{Tho68,TT19}, a discrete group is type I if and only if it is virtually abelian. Since the group $\bigoplus_{\mathbb{Z}_{<0}}F$ is not virtually abelian, it follows that it is not type I by Thoma's theorem. The result then follows.
\end{proof}

The purpose of the first subsection of this section is to provide a new and distinct countably infinite class of examples of non-type-I torsion locally compact contraction groups. The second subsection briefly discusses a future research focus for studying the unitary representation theory of torsion locally compact contraction groups. In particular, we discuss unipotent linear algebraic groups over $\mathbb{F}_p(\!(t)\!)$ and show that the $n$-dimensional Heisenberg groups over $\mathbb{F}_p(\!(t)\!)$ are CCR. 

\subsection{A new class of non-type-I torsion contraction groups}

\textit{We remind the reader of the following: throughout this subsection and the next, $p$ denotes a prime, and we identify the field $\mathbb{F}_p$ with $\mathbb{Z}/p\mathbb{Z}$ in the canonical way. If $x = \sum_{i=k}^\infty x_i t^i \in \mathbb{F}_p(\!(t)\!)$ with $x_k \ne 0$, we denote the integer $k$ by $\nu(x)$, and $x_i'$ will denote some coset representative of $x_i \in \mathbb{F}_p = \mathbb{Z}/p\mathbb{Z}$.}

Recall that it is shown in \cite{GW10} that the groups $(\bigoplus_{\mathbb{Z}_{<0}}F) \times (\prod_{\mathbb{Z}_{\ge 0}} F)$, for $F$ a finite simple group, are precisely the simple torsion locally compact contraction groups. It is also shown in \textit{loc.\ cit.} that every torsion locally compact contraction group $(N,\alpha)$ admits a finite length composition series

\begin{displaymath} 1 = N_0 \triangleleft N_1 \triangleleft \dots \triangleleft N_d = N \end{displaymath}

of $\alpha$-stable closed subgroups such that $N_{i+1}/N_{i}$ is isomorphic to a group of the form $(\bigoplus_{\mathbb{Z}_{<0}}F) \times (\prod_{\mathbb{Z}_{\ge 0}} F)$, for $F$ a finite simple group, for every $i$.

The groups given in Proposition \ref{prop:captors} are thus simple torsion locally compact contraction groups. The new class of examples of non-type-I torsion locally compact contraction groups provided in this subsection have composition length 2 as contraction groups and they contain an $\alpha$-stable closed normal central subgroup isomorphic to $\mathbb{F}_p(\!(t)\!)$ such that the quotient by it is also isomorphic $\mathbb{F}_p(\!(t)\!)$.

We now define these groups which were first introduced in \cite{GW21}. The groups are certain central extensions of $\mathbb{F}_p(\!(t)\!)$ by $\mathbb{F}_p(\!(t)\!)$. Recall that central extensions of an abelian group $A$ by another abelian group $B$ are described by 2-cocycles of $A$ with coefficients in $B$, and the isomorphism classes of such central extensions are in one-to-one correspondence with elements of the second cohomology group $H^2(A,B)$ (see, for example, \cite[Appendix A]{GW21} for details of this in the context of contraction groups, or \cite{ML95, Mor64} for more general considerations). It is shown in \cite{GW21} that those central extension of $\mathbb{F}_p(\!(t)\!)$ by $\mathbb{F}_p(\!(t)\!)$ which are contraction groups correspond precisely to those 2-cocycles which are \textit{equivariant} i.e.\ those 2-cocycles $\omega$ such that $\omega(tx,ty) = t\omega(x,y)$ for $x,y \in \mathbb{F}_p(\!(t)\!)$.


To define this class of groups, we write down a certain uncountably infinite class of continuous equivariant 2-cocycles that define contractive central extensions of $\mathbb{F}_p(\!(t)\!)$ by $\mathbb{F}_p(\!(t)\!)$. For the full details and proofs of the following, the reader should consult \cite[Section 7 $\&$ 8]{GW21}. We largely follow the notation used in this reference.

Fix $n \in \mathbb{N}$ and for $x := \sum_{i = \nu(x)}^{\infty} x_i t^i, y := \sum_{j = \nu(y)}^{\infty} y_j t^j \in \mathbb{F}_p(\!(t)\!)$, define
\begin{displaymath} \theta_n(x,y) := \sum_{i = \nu(x)}^{\infty} x_i y_{i+n} t^i. \end{displaymath}
With the same $x$ and $y$ as above, let $s \in \{0,1\}^{\mathbb{N}}$, and define
\begin{displaymath} \eta_s(x,y) := \sum_{n=1}^{\infty} s(n) t^n \theta_{2n}(x,y). \end{displaymath}

It is shown in \cite[Section 7]{GW21} that the $\eta_s$ are continuous equivariant bi-additive 2-cocycles on $\mathbb{F}_p(\!(t)\!)$ with coefficients in $\mathbb{F}_p(\!(t)\!)$. It is shown in \cite[Theorem 8.1]{GW21} that the corresponding central extensions, which we denote by $\mathbb{F}_p(\!(t)\!) \times_{\eta_s} \mathbb{F}_p(\!(t)\!)$ ($s \in \{0,1\}^{\mathbb{N}}$), are non-isomorphic for distinct $s \in \{0,1\}^\mathbb{N}$. Furthermore, they are contraction groups when equipped with the automorphism $\alpha$ which is multiplication by $t$ on each of the $\mathbb{F}_p(\!(t)\!)$ factors.

In this subsection we study the unitary representations of this class of central extensions. We show that if $s \in \{0,1\}^\mathbb{N}$ is non-zero and has finite support, then the corresponding central extension is not type I. 

Since we are working with groups that are a central extensions of $\mathbb{F}_p(\!(t)\!)$ by itself, it is natural to approach the study of their unitary representations by applying the Mackey little group method again. We now proceed with this analysis.

Fix $s \in \{0,1\}^\mathbb{N}$ non-zero and consider the group $\mathbb{F}_p(\!(t)\!) \times_{\eta_s} \mathbb{F}_p(\!(t)\!)$. For ease of notation, in what follows, let $G := \mathbb{F}_p(\!(t)\!) \times_{\eta_s} \mathbb{F}_p(\!(t)\!)$ and $N := \mathbb{F}_p(\!(t)\!)$ the canonical normal and central subgroup of $G$ isomorphic to $\mathbb{F}_p(\!(t)\!)$. First note that since $G$ is a central extension of the normal subgroup $N$, $G$ acts trivially on $\widehat{N}$. Thus, the singleton sets in $\widehat{N}$ are the orbits of the action $G \acts \widehat{N}$. Moreover, for a character $\chi \in \widehat{N}$, its stabiliser group is the whole group: $G_\chi = G$. Thus, to apply the Mackey little group method, and in particular Theorem \ref{thm:Mac2}, in this context, one must extend each character of $N$ to a character of $G$. However, it turns out that the only character of $N$ that we can extend to $G$ is the trivial character. We first prove a proposition and then provide some explanation for this.

\begin{prop}
Let $s \in \{0,1\}^\mathbb{N}$ be non-zero, $G= \mathbb{F}_p(\!(t)\!) \times_{\eta_s} \mathbb{F}_p(\!(t)\!)$ and $N = \mathbb{F}_p(\!(t)\!)$ the canonical normal central subgroup of $G$. Then $\overline{[G,G]} = N$.
\end{prop}

\begin{proof}
Let $(w,x), (y,z) \in G$ where $w,x,y,z \in \mathbb{F}_p(\!(t)\!)$. Then,
\begin{align*}
&[(w,x),(y,z)] = (w,x) (y,z)(w,x)^{-1}(y,z)^{-1} \\
&=(w,x) (y,z)(-w-\eta_s(x,-x),-x)(-y-\eta_s(z,-z),-z) \\
&= (w+y+\eta_s(x,z),x+z)(-(w+y)-\eta_s(x,-x)-\eta_s(z,-z)+\eta_s(-x,-z),-x-z) \\
&=(\eta_s(x,z)-\eta_s(x,-x)-\eta_s(z,-z)+\eta_s(-x,-z)+\eta_s(x+z,-x-z),0) 
\end{align*}
which is in $N$. So the commutator subgroup of $G$, and its closure, are contained in $N$. Conversely, we have that, by the above computation, for any $n \in \mathbb{Z}$ and $k \in \mathbb{N}$
\begin{align*}
&[(0,t^{n}),(0,t^{n+2k})] \\
&= (\eta_s(t^{n},t^{n+2k})+\eta_s(-t^{n},-t^{n+2k})+\eta_s(t^n+t^{n+2k},-t^n-t^{n+2k}),0) \\
&= (s(k)t^{n+k},0).
\end{align*}
Since $s$ is non zero, there exists a $k_0 \in \mathbb{N}$ such that $s(k_0)=1$, and it follows that $(t^{n+k_0},0) \in [G,G]$ for all $n \in \mathbb{Z}$. In particular, $(f,0) \in [G,G]$ for any polynomial $f$ in the $t^n$ ($n \in \mathbb{Z}$). This implies that $\overline{[G,G]} = N$.
\end{proof}

Now continuing the discussion from before the proposition: note that if we extend any character $\chi$ of $N$ to a character of $G$, it must be trivial on $\overline{[G,G]}$. But, since $\overline{[G,G]} = N$, and every character of $G$ must be trivial on $\overline{[G,G]}$, this implies that $\chi$ is trivial on $N$ and hence was originally the trivial character. In particular, only the trivial character of $N$ can be extended to a character of $G$.

To work around these issues, we must instead consider multiplier representations of $\mathbb{F}_p(\!(t)\!)$ and the corresponding Mackey little group method for multiplier representations. We now proceed with such analysis.

Identify $G/N$ with $\mathbb{F}_p(\!(t)\!)$. Define $\gamma: G/N \rightarrow G, y \mapsto (0,y)$, which is a measurable cross-section of the quotient map $q: G \rightarrow G/N$. For $\chi \in \widehat{N}$, define a function $\chi'$ on $G$ by $\chi'(x,y) = \chi(x)$. We then have the following result. We note that the overlines in the following lemma denote complex conjugation.

\begin{lem}
Let $s \in \{0,1\}^\mathbb{N}$ be non-zero, $G= \mathbb{F}_p(\!(t)\!) \times_{\eta_s} \mathbb{F}_p(\!(t)\!)$, $N = \mathbb{F}_p(\!(t)\!)$ the canonical normal central subgroup of $G$, and $\chi \in \widehat{N}$ non-zero. The extension $\chi'$ of $\chi$ to $G$ satisfies the following identity: $\chi'(g)\chi'(h) = \overline{\chi}(\eta_s(q(g),q(h)))\chi'(gh)$ for all $g,h \in G$. 
\end{lem} 

\begin{proof}
For $x,y \in \mathbb{F}_p(\!(t)\!)$, let $\beta(x,y) := \gamma(xy)^{-1} \gamma(x) \gamma(y)$. Note that, by definition of the central extension, 
\begin{align*}
\beta(x,y) = (0,xy)&^{-1}(0,x)(0,y) = (-\eta_s(xy,-xy), -xy)(0,x)(0,y) \\
&= (-\eta_s(xy,-xy), -xy)(\eta_s(x,y),xy) = (\eta_s(x,y),0).
\end{align*}
It follows by \cite[Lemma 4.47]{KT12} and the discussion preceding the lemma that for all $g,h \in G$,
$\chi'(g)\chi'(h) = \overline{\chi'}(\beta(q(g),q(h))\chi'(gh) = \overline{\chi}(\eta_s(q(g),q(h))) \chi'(gh).$\end{proof}

Recall that the characters on $\mathbb{F}_p(\!(t)\!)$ are indexed by the elements of $\mathbb{F}_p(\!(t)\!)$ itself and we denote the character corresponding to an element $z \in \mathbb{F}_p(\!(t)\!)$ by $\chi_z$ (see Proposition \ref{prop:charfpt}). Define on $G \times G$ a function $\tilde\omega_{(s,z)}$ given by $\tilde\omega_{(s,z)}(g,h) := \overline{\chi_{-z}}(\eta_s(q(g),q(h)))$. The function $\tilde\omega_{(s,z)}$ is a multiplier on $\mathbb{F}_p(\!(t)\!) \times_{\eta_s} \mathbb{F}_p(\!(t)\!)$; see \cite[Lemma 4.48]{KT12} for the exact proof if required. It is called the \textit{Mackey obstruction} corresponding to the character $\chi_{-z}$. The following proposition follows directly from the previous lemma.

\begin{prop}
Let $s \in \{0,1\}^\mathbb{N}$, $G= \mathbb{F}_p(\!(t)\!) \times_{\eta_s} \mathbb{F}_p(\!(t)\!)$ and $N = \mathbb{F}_p(\!(t)\!)$ the canonical normal central subgroup of $G$. For any $z \in N$, the representation $\chi_{-z}'$ is a $\tilde\omega_{(s,z)}$-representation of $G$.
\end{prop}

Since the multiplier $\tilde\omega_{(s,z)}$ is constant on $N$ cosets and $G/N \cong \mathbb{F}_p(\!(t)\!)$, we may identify $\tilde\omega_{(s,z)}$ as being lifted from the multiplier $\omega_{(s,z)}$ on $\mathbb{F}_p(\!(t)\!)$ defined by $\omega_{(s,z)}(x,y) := \tilde\omega_{(s,z)}((0,x),(0,y))$ for $x,y \in \mathbb{F}_p(\!(t)\!)$. We then get the following method for showing that the groups $\mathbb{F}_p(\!(t)\!) \times_{\eta_s} \mathbb{F}_p(\!(t)\!)$ are not type I.

\begin{prop}
Let $s \in \{0,1\}^\mathbb{N}$ be non-zero, $G= \mathbb{F}_p(\!(t)\!) \times_{\eta_s} \mathbb{F}_p(\!(t)\!)$ and $N = \mathbb{F}_p(\!(t)\!)$ the canonical normal central subgroup of $G$. Let $z \in \mathbb{F}_p(\!(t)\!)$ and $\sigma$ a factor $\omega_{(s,-z)}^{-1}$-representation of $G/N \cong \mathbb{F}_p(\!(t)\!)$. Then, $\chi_z' \otimes \tilde\sigma$, where $\tilde\sigma$ denotes the lift of $\sigma$ to $G$, is a factor unitary representation of $G$, and its type is the same as the type of $\sigma$.
\end{prop}

\begin{proof}
Since $\chi_z'$ is a $\tilde{\omega}_{(s,-z)}$-representation of $G$, and $\tilde\sigma$ is a $\tilde\omega_{(s,-z)}^{-1}$-representation of $G$, it follows that $\chi_z' \otimes \tilde\sigma$ is a $\tilde{\omega}_{(s,-z)}\tilde\omega_{(s,-z)}^{-1}$-representation of $G$ \cite[$\S$1]{Mac58}. Since $\tilde{\omega}_{(s,-z)}\tilde\omega_{(s,-z)}^{-1}$ is the trivial multiplier, $\chi_z' \otimes \tilde\sigma$ is in fact a unitary representation of $G$. 

Let $\tilde\sigma(G)''$ (resp.\ $\sigma(G/N)''$) denote the von Neumann algebra generated by $\tilde\sigma$ (resp.\ $\sigma$). The von Neumann algebra generated by $\chi_z' \otimes \tilde\sigma$ is $\mathbb{C} \otimes \tilde\sigma(G)''$ (here, the later tensor product is the tensor product of von Neumann algebras in the sense of \cite[$\S$2.4]{Dix81}). Also, by definition of $\sigma$ and $\tilde\sigma$, $\sigma(G/N)$ and $\tilde\sigma(G)$ generate the same von Neumann algebra, so $\mathbb{C} \otimes \tilde\sigma(G)''$ and $\mathbb{C} \otimes \sigma(G/N)''$ are isomorphic. The fact that $\chi_z' \otimes \tilde\sigma$ is a factor representation and has the same type as $\sigma$ then follows from basic results on von Neumann algebras \cite[$\S$2]{Dix81}.
\end{proof}

\begin{cor}
Let $s \in \{0,1\}^\mathbb{N}$ be non-zero and $G= \mathbb{F}_p(\!(t)\!) \times_{\eta_s} \mathbb{F}_p(\!(t)\!)$. If there exists $z \in \mathbb{F}_p(\!(t)\!)$ such that $\omega_{(s,z)}$ is not a type I multiplier, then $G$ is not type I.
\end{cor}

\begin{proof}
Take $\sigma$ to be a non-type-I factor $\omega_{(s,z)}$-representation of $G/N \cong \mathbb{F}_p(\!(t)\!)$ in the previous proposition. Then $\chi'_{-z} \otimes \tilde\sigma$ is a non-type-I factor representation of $G$.
\end{proof}

We now give an explicit description of the multipliers $\omega_{(s,z)}$ and work towards determining which of these multipliers are type I. Note that for $x = \sum_{j = \nu(x)}^{\infty} x_j t^j, y = \sum_{j = \nu(y)}^{\infty} y_j t^j \in \mathbb{F}_p(\!(t)\!)$, by definition of the cocyle $\eta_s$, we have 
\begin{displaymath} \eta_s(x,y) = \sum_{n \in \supp(s)} \sum_{j=\nu(x)}^{\infty} x_j y_{j+2n} t^{j+n}. \end{displaymath}
Then,
\begin{align*}
\omega_{(s,z)}(x,y) = &\chi_z(\eta_s(x,y)) = \prod_{n \in \supp(s)} \prod_{j=\nu(x)}^{\infty} \chi_z(x_jy_{j+2n}t^{j+n}) \\
&= \prod_{n \in \supp(s)} \prod_{j=\nu(x)}^\infty \exp\bigg(\frac{2\pi i x_j' y_{j+2n}'z_{-j-n}'}{p}\bigg) \\
&= \exp\bigg(\frac{2\pi i}{p} \sum_{n \in \supp(s)} \sum_{j=\nu(x)}^\infty x_j'y_{j+2n}'z_{-j-n}'\bigg).
\end{align*}

Note that if $j > -\nu(z) - n$, then $z_{-j-n} = 0$. Thus it follows that 
\begin{align*}  \omega_{(s,z)}&(x,y) = \prod_{n \in \supp(s)} \prod_{j=\nu(x)}^{-\nu(z)-n} \exp\bigg(\frac{2\pi i x_j' y_{j+2n}'z_{-j-n}'}{p}\bigg) \\
&= \exp\bigg(\frac{2\pi i}{p} \sum_{n \in \supp(s)} \sum_{j=\nu(x)}^{-\nu(z)-n} x_j'y_{j+2n}'z_{-j-n}'\bigg).
\end{align*}


Recall from the preliminaries section of the article, that given a multiplier $\omega$ on a locally compact abelian group $G$, we define a group $S_\omega$ by $S_\omega = \{ x \in G : \omega^{(2)}(x,y)=1 \; \forall y \in G \}$ where $\omega^{(2)}(x,y) = \omega(x,y) \omega(y,x)^{-1}$.

\begin{prop}\label{prop:prop5.3}
Let $z \in \mathbb{F}_p(\!(t)\!)$, $s \in \{0,1\}^{\mathbb{N}}$ be non-zero, and suppose that $s$ is finitely supported. Then the subgroup $S_{\omega_{(s,z)}}$ contains a compact open subgroup.
\end{prop}

\begin{proof}
Assume the hypotheses of the proposition. For $x = \sum_{j = \nu(x)}^{\infty} x_j t^j, y = \sum_{j = \nu(y)}^\infty y_j t^j \in \mathbb{F}_p(\!(t)\!)$,
\begin{displaymath}  \omega_{(s,z)}(x,y) = \exp\bigg(\frac{2\pi i}{p} \sum_{n \in \supp(s)} \sum_{j=\nu(x)}^{-\nu(z)-n} x_j' y_{j+2n}' z_{-j-n}'\bigg) \end{displaymath}
and
\begin{align*}  
\omega_{(s,z)}(y,x) = \exp\bigg(\frac{2\pi i}{p} &\sum_{n \in \supp(s)} \sum_{j=\nu(y)}^{-\nu(z)-n} y_j' x_{j+2n}' z_{-j-n}'\bigg) \\
&= \exp\bigg(\frac{2\pi i}{p} \sum_{n \in \supp(s)} \sum_{j=\nu(x)-2n}^{-\nu(z)-n} x_{j+2n}'y_j' z_{-j-n}'\bigg). 
\end{align*}

Thus,
\begin{align*} \omega_{(s,z)}^{(2)}(x,y) = \exp\bigg( \frac{2\pi i}{p} \bigg( \sum_{n \in \supp(s)} &\sum_{j=\nu(x)}^{-\nu(z)-n} (x_j' y_{j+2n}' - x_{j+2n}'y_j')z_{-j-n}' \\
& - \sum_{n \in \supp(s)} \sum_{j=\nu(x)-2n}^{\min\{-\nu(z)-n,\nu(x)-1\}} x_{j+2n}' y_j' z_{-j-n}' \bigg)\bigg). 
\end{align*}

Let $m := \max \supp(s)$, which exists since we have assumed that $s$ is non-zero and $\supp(s)$ is finite. If we choose $x$ so that $\nu(x) -2m > -\nu(z) - m$, or equivalently, $\nu(x) > -\nu(z) + m$, then there are no terms in the above sums and hence it follows that $\omega_{(s,z)}^{(2)}(x,y) = 1$ for all $y \in \mathbb{F}_p(\!(t)\!)$. In particular, $S_{\omega_{(s,z)}}$ contains the compact open subgroup $U_{-\nu(z)+m} := \{ x \in \mathbb{F}_p(\!(t)\!) : \nu(x) > -\nu(z) + m \}$.
\end{proof}

\begin{cor}
Let $z \in \mathbb{F}_p(\!(t)\!)$, $s \in \{0,1\}^{\mathbb{N}}$ be non-zero, and suppose that $s$ has finite support. Then $\omega_{(s,z)}$ is type I if and only if $\mathbb{F}_p(\!(t)\!)/S_{\omega_{(s,z)}}$ is finite.
\end{cor}

\begin{proof}
Assume the hypotheses. By the previous proposition, $S_{\omega_{(s,z)}}$ contains a compact open subgroup, thus it is itself open. This implies that the quotient $\mathbb{F}_p(\!(t)\!)/S_{\omega_{(s,z)}}$ is discrete.

By \cite[Theorem 3.1]{BK73}, $\omega_{(s,z)}$ is similar to a multiplier $\omega$ on $\mathbb{F}_p(\!(t)\!)$ such that $\omega$ is lifted from a totally-skew multiplier $\omega_1$ on $\mathbb{F}_p(\!(t)\!)/S_{\omega_{(s,z)}}$, and furthermore, $\omega_{(s,z)}$ is type I if and only if $\omega$ is type I if and only if $\omega_1$ is type I. By \cite[Lemma 3.1]{BK73}, the multiplier $\omega_1$ is type I if and only if $\mathbb{F}_p(\!(t)\!)/S_{\omega_\chi}$ is finite.
\end{proof}

\begin{thm}\label{thm:nontypeigps}
Let $s \in \{0,1\}^{\mathbb{N}}$ be non-zero and suppose that $s$ is finitely supported. The group $\mathbb{F}_p(\!(t)\!) \times_{\eta_s} \mathbb{F}_p(\!(t)\!)$ is not type I.
\end{thm}

\begin{proof}
Let $k_0 \in \mathbb{N}$ and define $z_0 := \sum_{j = k_0}^{\infty} t^j \in \mathbb{F}_p(\!(t)\!)$. Set $m := \max \supp(s)$. One computes, similar to a prior calculation, that for $x = \sum_{j = \nu(x)}^{\infty} x_j t^j, y = \sum_{j = \nu(y)}^{\infty} y_j t^j \in \mathbb{F}_p(\!(t)\!)$,
\begin{align*} \omega_{(s,z_0)}^{(2)}(x,y) = \exp\bigg(\frac{2\pi i}{p} \bigg( \sum_{n \in \supp(s)} &\sum_{j=\nu(x)}^{-k_0-n} (x_j' y_{j+2n}' - x_{j+2n}'y_j') \bigg)\\
& - \sum_{n \in \supp(s)} \sum_{j=\nu(x)-2n}^{\min\{-k_0-n,\nu(x)-1\}} x_{j+2n}' y_j' \bigg). 
\end{align*}
By the proof of Proposition \ref{prop:prop5.3}, for any $x \in U_{-k_0+m} = \{ x \in \mathbb{F}_p(\!(t)\!) : \nu(x) > -k_0 + m \}$, $\omega_{(s,z)}^{(2)}(x,y)=1$ for all $y \in \mathbb{F}_p(\!(t)\!)$. Conversely, for any non-zero $x = \sum_{j = \nu(x)}^{\infty} x_j t^j \in \mathbb{F}_p(\!(t)\!)$ with $x \notin U_{-k_0+m}$ (i.e. $\nu(x) \le -k_0 + m$), we may choose $a \in \mathbb{F}_p$ so that $ax_{\nu(x)} \ne 0$ in $\mathbb{F}_p$. It then follows that $y := at^{\nu(x) - 2m} \in \mathbb{F}_p(\!(t)\!)$ satisfies 
\begin{displaymath} \omega_{(s,z_0)}^{(2)}(x,y) = \exp\bigg(\frac{2\pi i}{p} a'x_{\nu(x)}'\bigg) \ne 1
\end{displaymath}
so $x \notin S_{\omega_{(s,z_0)}}$.
This thus implies that $S_{\omega_{(s,z_0)}} = U_{-k_0 + m}$. In particular, the quotient $\mathbb{F}_p(\!(t)\!)/S_{\omega_{(s,z_0)}}$ is infinite and discrete, so $\omega_{(s,z_0)}$ is not a type I multiplier by the previous corollary.
\end{proof}

One should note that, in the proof of the above theorem, we determined exactly what the group $S_{\omega_{(s,z_0)}}$ is when $s$ has finite support. We state this as a proposition below to bring it to the attention of the reader.

\begin{prop}\label{prop:propkernel}
Define $z_0 := \sum_{i = k_0}^{\infty} t^i \in \mathbb{F}_p(\!(t)\!)$, let $s \in \{0,1\}^\mathbb{N}$ be non-zero and finitely supported, and set $m:=\max\supp(s)$. Then $S_{\omega_{(s,z_0)}} = U_{-k_0+m} = \{ x \in \mathbb{F}_p(\!(t)\!) : \nu(x) > -k_0 + m \}$.
\end{prop}

At the time of writing this article, we cannot workout exactly what happens when $s \in \{0,1\}^\mathbb{N}$ does not have finite support. A follow up article to this one is studying the groups $\mathbb{F}_p(\!(t)\!) \times_{\eta_s} \mathbb{F}_p(\!(t)\!)$ for when $s$ is not finitely supported and determining when these groups are type I. This follow up article uses more algebraic techniques than what is contained in the present article.


It would still be interesting to us to solve this type I problem using the multiplier approach presented here as well. We thus leave some open questions that we cannot solve at this stage.

At the moment, it appears that the group $S_{\omega_{(s,z_0)}}$ is trivial whenever the support of $s$ is not finite, and Proposition \ref{prop:propkernel} provides support for this. We would like to know what this group is for any $s$ as it is usefully in determining which of these multipliers are type I.

\begin{prob}
Define $z_0 := \sum_{i = k_0}^{\infty} t^i \in \mathbb{F}_p(\!(t)\!)$ and let $s \in \{0,1\}^\mathbb{N}$ be non-finitely supported. Is $S_{\omega_{(s,z_0)}}$ trivial? If not, what is this group?
\end{prob}

We will now give a more general outline of how to determine whether a multiplier is type I or not.

It can be shown that for any locally compact abelian group $G$ and multiplier $\omega$ on $G$, the function $\omega^{(2)}$ is a bicharacter on $G$ i.e.\ is a character when you keep one coordinate fixed and vary the other coordinate. It thus follows that you get a homomorphism

\begin{displaymath} h_\omega: G \rightarrow \widehat{G}, x \mapsto \omega^{(2)}(x, \cdot) \end{displaymath}

of locally compact abelian groups. It is shown in \cite[Page 310]{BK73} that the multiplier $\omega$ is type $\I$ if and only if the map $h_\omega$ is an open map of $G$ into $S_\omega^\perp := \{ \chi \in \widehat{G} : \chi(S_\omega) = \{1\}\}$. Furthermore, if $\omega$ is totally skew (i.e.\ $S_\omega$ is trivial), then $\omega$ is type $\I$ if and only if $h_\omega$ is bicontinuous. As is noted on page 310 in \cite{BK73}, if $G$ is also $\sigma$-compact, in order to show that $\omega$ is type $\I$, it suffices to show that the image of $h_\omega$ is closed in $\widehat{G}$, or that $h_\omega$ is an open map to $\widehat{G}$. If $\omega$ is totally-skew, then the map $h_\omega$ is injective and has dense range in $\widehat{G}$.

By considering the above results, we put forward the following problem for determining whether the multipliers $\omega_{(s,z_0)}$ of $\mathbb{F}_p(\!(t)\!)$ are type $\I$ or non-type-$\I$ for arbitrary $s$.

\begin{prob}
Let $z_0 = \sum_{i = k_0}^{\infty} t^i$ and $s \in \{ 0,1\}^\mathbb{N}$ be non-zero. Determine whether the map
\begin{displaymath}h_{\omega_{(s,z_0)}}: \mathbb{F}_p(\!(t)\!) \rightarrow S_{\omega_{(s,z_0)}}^\perp \le \widehat{\mathbb{F}_p(\!(t)\!)}, x \mapsto \omega_{(s,z_0)}^{(2)}(x, \cdot) \end{displaymath}
is, equivalently, (not) open, (not) surjective or has (non-)closed image.
\end{prob}

Remark that for this particular $z_0 = \sum_{j = k_0}^{\infty} t^j$ and $x = \sum_{j = \nu(x)}^{\infty} x_j t^j, y = \sum_{j = \nu(y)}^{\infty} y_j t^j \in \mathbb{F}_p(\!(t)\!)$, the bicharacter $\omega_{(s,z_0)}^{(2)}(x,y)$ is given by the complex number
\begin{align*} &\prod_{n \in \supp(s)} \bigg( \prod_{j=\nu(x)}^{-k_0-n} \exp\bigg(\frac{2\pi i (x_j' y_{j+2n}' - x_{j+2n}'y_j')}{p} \bigg) \cdot\prod_{j=\nu(x)-2n}^{\min\{-k_0-n,\nu(x)-1\}} \exp\bigg(-\frac{2\pi i x_{j+2n}' y_j'}{p}\bigg) \bigg)\\
&=\exp\bigg( \frac{2\pi i}{p} \bigg( \sum_{n \in \supp(s)} \bigg(\sum_{j=\nu(x)}^{-k_0-n} (x_j' y_{j+2n}' - x_{j+2n}'y_j') - \sum_{j=\nu(x)-2n}^{\min\{-k_0-n,\nu(x)-1\}} x_{j+2n}' y_j' \bigg)\bigg)\bigg).
\end{align*}

Even more generally, we also put forward the following problem, as we do not have a solution to this at the current time.

\begin{prob}
Does there exist uncountably many non-type-I torsion locally compact contraction groups?
\end{prob}

\subsection{Unipotent linear algebraic groups over $\mathbb{F}_p(\!(t)\!)$}

In pursuit of determining the type $\I$ torsion locally compact contraction groups, we propose that the next natural class of torsion locally compact contraction groups to investigate for this problem are the class of unipotent linear algebraic groups over $\mathbb{F}_p(\!(t)\!)$. We state this as an open problem.

\begin{prob}
Determine which unipotent linear algebraic groups over $\mathbb{F}_p(\!(t)\!)$ are type $\I$ or CCR.
\end{prob}

This problem is also alluded to in the recent paper \cite{BE21} where it is shown that linear algebraic groups over local fields of characteristic zero are type $\I$. There has also been previous efforts in working towards a solution to this question \cite{How77,How77a,EK12}. In \cite{EK12}, the authors generalise the Kirillov orbit method to a quite general class of nilpotent locally compact groups, which includes, for example, the group of $n$-dimensional unipotent matrices over $\mathbb{F}_p(\!(t)\!)$, and they provide methods for determining when such groups are type I or CCR. However, the above question is still open.

To conclude this article, we provide a proof that the $(2n+1)$-dimensional Heisenberg group over $\mathbb{F}_p(\!(t)\!)$, which we will denote by $\mathbb{H}_n(\mathbb{F}_p(\!(t)\!))$, is CCR.

These Heisenberg groups are one of the most basic and well known classes of unipotent linear algebraic groups and are a natural class of groups to begin looking at for the above problem. Following the techniques used in this article, we apply the Mackey little group method again to study the unitary representations of these groups.

By the group $\mathbb{H}_n(\mathbb{F}_p(\!(t)\!))$, we mean the group of matrices of the form 

$$\begin{pmatrix}
1 & x_1 & \cdot & \cdot & \cdot & x_{n} & z \\
   & 1    & 0        & \cdot & \cdot & 0          & y_1 \\
   &       & \cdot &     \cdot      &           & \cdot    & \cdot \\
   &       &          & \cdot  &     \cdot      & \cdot    & \cdot \\
   &       &          &            & \cdot & 0           & \cdot \\
   &  \text{{\Huge 0}}     &         &            &           & 1           & y_{n} \\
   &       &          &            &           &              & 1 \\
\end{pmatrix}$$

where each of the $x_i,y_i$ and $z$ are in $\mathbb{F}_p(\!(t)\!)$. As a set, $\mathbb{H}_n(\mathbb{F}_p(\!(t)\!))$ is isomorphic to $\mathbb{F}_p(\!(t)\!)^n \times \mathbb{F}_p(\!(t)\!)^n \times \mathbb{F}_p(\!(t)\!)$, so from now on, we will denote a matrix in $\mathbb{H}_n(\mathbb{F}_p(\!(t)\!))$ as above by $(\xi,\upsilon,z) \in \mathbb{F}_p(\!(t)\!)^n \times \mathbb{F}_p(\!(t)\!)^n \times \mathbb{F}_p(\!(t)\!)$, where $\xi$ and $\upsilon$ are $n$-dimensional vectors whose entries are the $x_i$ and $y_i$ respectively. The product on $\mathbb{H}_n(\mathbb{F}_p(\!(t)\!))$ is then given by
\begin{displaymath} (\xi,\upsilon,z)(\zeta,\mu,w) = (\xi+\zeta, \upsilon+\mu, z+w+\xi \cdot \mu)  \end{displaymath}
where $\xi \cdot \mu$ denotes the dot product of the vectors $\xi$ and $\mu$.

Using this notation for the group $\mathbb{H}_n(\mathbb{F}_p(\!(t)\!))$, we then have that $N := \{ (0,\upsilon,z) \in \mathbb{H}_n(\mathbb{F}_p(\!(t)\!)): \upsilon \in \mathbb{F}_p(\!(t)\!)^n, z \in \mathbb{F}_p(\!(t)\!)\}$ is a normal subgroup of $\mathbb{H}_n(\mathbb{F}_p(\!(t)\!))$, and $\mathbb{H}_n(\mathbb{F}_p(\!(t)\!))$ is a semi-direct product of $N$ with the subgroup $A := \{ (\xi,0,0) \in \mathbb{H}_n(\mathbb{F}_p(\!(t)\!)): \xi \in \mathbb{F}_p(\!(t)\!)^n\}$.

We now prove that these Heisenberg groups are CCR.

\begin{thm}
Let $n \in \mathbb{N}$ and $p$ a prime. Define $N := \{ (0,\upsilon,z) \in \mathbb{H}_n(\mathbb{F}_p(\!(t)\!)): \upsilon \in \mathbb{F}_p(\!(t)\!)^n, z \in \mathbb{F}_p(\!(t)\!)\}$ and $A := \{ (\xi,0,0) \in \mathbb{H}_n(\mathbb{F}_p(\!(t)\!)): \xi \in \mathbb{F}_p(\!(t)\!)^n\}$. The following hold:
\begin{enumerate}[(i)]
   \item The action $A \acts \widehat{N}$ has closed orbits; \
   \item $N$ is regularly embedded in $\mathbb{H}_n(\mathbb{F}_p(\!(t)\!))$; \
   \item $\mathbb{H}_n(\mathbb{F}_p(\!(t)\!))$ is CCR.
\end{enumerate}
\end{thm}

\begin{proof}
Note that $N$ is isomorphic to $\mathbb{F}_p(\!(t)\!)^{n+1}$ as abelian groups, and, by work earlier in this article, we have that $\widehat{N} \cong \widehat{\mathbb{F}_p(\!(t)\!)}^{n+1} \cong \mathbb{F}_p(\!(t)\!)^{n+1} \cong N$. Recall, also, that the characters on $\mathbb{F}_p(\!(t)\!)$ are indexed by the elements of $\mathbb{F}_p(\!(t)\!)$ itself. The character corresponding to an element $z = \sum_{j = \nu(z)}^{\infty} z_j t^j \in \mathbb{F}_p(\!(t)\!)$ is denoted by $\chi_z$ and is defined on an element $x = \sum_{j = \nu(x)}^{\infty} x_j t^j$ by 

\begin{displaymath} \chi_z(x) = \prod_{j \in \mathbb{Z}} \exp\bigg(\frac{2\pi i x_j' y_{-j}'}{p}\bigg). \end{displaymath}

Then, the character in $\widehat{N}$ corresponding to an element $(0,\upsilon,z) \in N$ will be denoted by $\chi_{\upsilon,z}$ and is defined on an element $(0,\mu,w) \in N$ by 

\begin{displaymath} \chi_{\upsilon,z}(0,\mu,w) = \chi_{y_1}(m_1) \cdots \chi_{y_{n}}(m_{n})\chi_z(w) \end{displaymath}

where the $y_i$ and $m_i$ are the components of the vectors $\upsilon$ and $\mu$ respectively.

We now compute the action of $A$ on $\widehat{N}$. Note that for $(\xi,0,0) \in A$ and $(0,\mu,w) \in N$, we have that 

\begin{displaymath} (\xi,0,0)(0,\mu,w)(-\xi,0,0) = (\xi,\mu,w+\xi\cdot\mu)(-\xi,0,0) = (0,\mu,w+\xi\cdot\mu)   \end{displaymath}

from which it follows that

\begin{align*}
(\xi,0,0) \cdot \chi_{\upsilon,z}(0,\mu,w) &= \chi_{\upsilon,z}(0,\mu,w+\xi\cdot\mu) \\
&= \chi_{y_1}(m_1)\cdots\chi_{y_n}(m_n)\chi_z(w + x_1m_1 + \cdots + x_n m_n) \\
&= \chi_{y_1}(m_1)\cdots\chi_{y_n}(m_n)\chi_z(w)\chi_z(x_1m_1) \cdots \chi_z(x_n m_n). \\
\end{align*}

Now, for $a = \sum_{j = \nu(a)}^{\infty} a_j t^j, b = \sum_{j = \nu(b)}^{\infty} b_j t^j, c = \sum_{j = \nu(c)}^{\infty} c_j t^j \in \mathbb{F}_p(\!(t)\!)$, we have that

\begin{align*}
\chi_c(ab) &= \prod_{j \in \mathbb{Z}} \exp\bigg(\frac{2\pi i c_j' (ab)_{-j}'}{p}\bigg) \\
&= \prod_{j \in \mathbb{Z}} \exp\bigg(\frac{2\pi i c_j' (\sum_{k \in \mathbb{Z}} b_k' a_{-j-k}')}{p}\bigg) \\
&= \exp\bigg(\frac{2\pi i (\sum_{j,k \in \mathbb{Z}} c_j' b_k' a_{-j-k}')}{p}\bigg) \\
&= \prod_{k \in \mathbb{Z}} \exp\bigg(\frac{2\pi i b_k' (\sum_{j \in \mathbb{Z}} c_j' a_{-j-k}')}{p}\bigg) \\
&= \prod_{k \in \mathbb{Z}} \exp\bigg(\frac{2\pi i b_k' (ac)_{-k}'}{p}\bigg) \\
&= \prod_{k \in \mathbb{Z}} \exp\bigg(\frac{2\pi i (ac)_{k}'b_{-k}' }{p}\bigg) \\
&= \chi_{ac}(b).
\end{align*}

We do not have to worry about convergence of the above products and sums since there are only finitely many non-zero terms in each line.

Hence,

\begin{align*}
(\xi,0,0) \cdot \chi_{\upsilon,z}(0,\mu,w) &= \chi_{y_1}(m_1)\cdots\chi_{y_n}(m_n)\chi_z(w)\chi_z(x_1m_1) \cdots \chi_z(x_n m_n) \\
&= \chi_{y_1}(m_1)\cdots\chi_{y_n}(m_n)\chi_z(w)\chi_{x_1z}(m_1) \cdots \chi_{x_nz}(m_n) \\
&= \chi_{y_1 + x_1z}(m_1)\cdots\chi_{y_n +x_nz}(m_n)\chi_z(w) \\
&= \chi_{\upsilon + z\xi,z}(0,\mu,w).
\end{align*}

So, the dual action of $A$ on $N$, that is, the action of $A$ on $N$ arising from the action of $A$ on $\widehat{N}$ and the identification of $\widehat{N}$ with $N$, is given by $(\xi,0,0) \cdot_{\text{dual}} (0,\upsilon,z) := (0,\upsilon + z \xi, z)$. If $z \ne 0$, then the dual-orbit of the point $(0,\upsilon,z)$ is precisely $A\cdot_{\text{dual}} (0,\upsilon,z) = \{ (0,\upsilon,z) \in N : \upsilon \in \mathbb{F}_p(\!(t)\!) \}$, and if $z=0$, the dual-orbit of $(0,\upsilon,0)$ is just a point. This discussion shows that the orbits of the dual action of $A$ on $N$ are closed and this proves (i). By Proposition \ref{prop:regemb}, it follows that the group $N$ is regularly embedded in $\mathbb{H}_n(\mathbb{F}_p(\!(t)\!))$, and this proves (ii). To prove (iii), note that $C^*(\mathbb{H}_n(\mathbb{F}_p(\!(t)\!)))$ is isomorphic to $C_0(N) \rtimes_\alpha A$ by Proposition \ref{prop:abcrossprod} and the fact that $N\cong\widehat{N}$, where $\alpha$ is the automorphism arising from the dual action of $A$ on $N$. Theorem \ref{thm:gcrccrcross} then implies that $C^*(\mathbb{H}_n(\mathbb{F}_p(\!(t)\!))) \cong C_0(N) \rtimes_\alpha A$ is CCR since the orbits of the dual action of $A$ on $N$ are closed, so (iii) follows. 
\end{proof}

\begin{cor}
There exist non-abelian torsion CCR locally compact contraction groups.
\end{cor}

\begin{center}
\textsc{Acknowledgements}
\end{center}
The author would like to thank George Willis, Colin Reid and Stephan Tornier for useful discussions during his Master's studies at The University of Newcastle during which the majority of this work took place. He would also like to thank George Willis, Pierre-Emmanuel Caprace and the anonymous referee for many useful comments on previous versions of this manuscript.

\bibliographystyle{amsplain}
\bibliography{con_rep_bib}


\end{document}